\documentclass[onefignum,onetabnum]{siamart171218}

\usepackage{lipsum}
\usepackage{amsfonts,amssymb}
\usepackage{graphicx}
\usepackage{epstopdf}
\usepackage{algorithmic}
\ifpdf
  \DeclareGraphicsExtensions{.eps,.pdf,.png,.jpg}
\else
  \DeclareGraphicsExtensions{.eps}
\fi

\newcommand{\R}{\mathbb{R}}

\newcommand{\eps}{\epsilon}

\newcommand{\e}{\varepsilon}

\newcommand{\la}{\lambda}

\newcommand{\ds}{\displaystyle}

\newcommand{\uep}{u_{\e}}
\newcommand{\nep}{n_{\e}}

\newcommand{\dd}{\mathrm{d}}

\newcommand{\ulbar}{\underline u}

\DeclareMathOperator{\Int}{Int}

\DeclareMathOperator{\sign}{sign}
\renewcommand{\epsilon}{\varepsilon}

\newsiamremark{remark}{Remark}
\newsiamremark{hypothesis}{Hypothesis}
\crefname{hypothesis}{Hypothesis}{Hypotheses}
\newsiamthm{claim}{Claim}
\newsiamthm{hyp}{hypothesis}
\newsiamthm{thm}{Theorem}
\newsiamthm{lem}{Lemma}

\newsiamthm{example}{Example}

\headers{Thin front of integro-diff KPP equation with fat-tailed kernels}{E. Bouin, J.Garnier, C.Henderson and F. Patout}

\title{Thin front limit of an integro--differential Fisher--KPP equation with fat--tailed kernels\thanks{Submitted to the editors DATE.
\funding{Part of this work was performed within the framework of the LABEX MILYON (ANR- 10-LABX-0070) of Universit\'e de Lyon, within the program ``Investissements d'Avenir'' (ANR-11- IDEX-0007) and the project NONLOCAL (ANR-14-CE25-0013) operated by the French National Research Agency (ANR). In addition, this project has received funding from the European Research Council (ERC) under the European Unions Horizon 2020 research and innovation programme (grant agreement No 639638).  CH was partially supported by the National Science Foundation Research Training Group grant
DMS-1246999. }}}

\author{Emeric Bouin\thanks{Universit\'e Paris-Dauphine, UMR CNRS 7534,Place du Mar\'echal de Lattre de Tassigny, 75775 Paris Cedex 16, France.
  (\email{bouin@ceremade.dauphine.fr}, \url{https://www.ceremade.dauphine.fr/\string ~bouin/}).}
\and Jimmy Garnier \thanks{Univ. Savoie Mont--Blanc, CNRS, LAMA, F-73000 Chamb\'ery, France.
  (\email{jimmy.garnier@univ-smb.fr},\url{http://www.lama.univ-savoie.fr/\string ~garnier/})}
\and Christopher Henderson \thanks{Department of Mathematics, the University of Chicago, 5734 S.~University Ave., Chicago, Il 60637.
  (\email{henderson@math.uchicago.edu},\url{https://math.uchicago.edu/\string ~henderson/})}
\and Florian Patout \thanks{UMPA, Ecole Normale Sup\'erieure de Lyon, UMR CNRS 5669 , 46 all\'ee d'Italie, F-69364~Lyon~cedex~07, France.
  (\email{florian.patout@ens-lyon.fr},\url{http://perso.ens-lyon.fr/florian.patout/})}
}

\usepackage{amsopn}

\ifpdf
\hypersetup{
  pdftitle={Thin front limit of an  integro--differential Fisher--KPP equation with fat--tailed kernels},
  pdfauthor={E. Bouin, J.Garnier, C.Henderson and F. Patout}
}
\fi

\begin{document}

\maketitle

\begin{abstract}
 We study the asymptotic behavior of solutions to a monostable integro-differential Fisher-KPP equation, that is where the standard Laplacian is replaced by a convolution term, when the dispersal kernel is fat-tailed. We focus on two different regimes. Firstly, we study the long time/long range scaling limit by introducing a relevant rescaling in space and time and prove a sharp bound on the (super-linear) spreading rate in the Hamilton-Jacobi sense by means of sub- and super- solutions. Secondly, we investigate a long time/small mutation regime for which, after identifying a relevant rescaling for the size of mutations, we derive a Hamilton-Jacobi limit.
\end{abstract}



%
%

\section{Introduction}
\subsection*{The model}

In this paper, we focus on the asymptotic behavior of solutions to the following integro--differential equation 
\begin{equation}\label{eq:main}
\left \lbrace\begin{array}{ll}
	n_t = J * n-n + n(1 - n), & \text{ in } \  (0,+\infty) \times \R, \medskip\\
	n(t=0,\cdot) = n^0.
\end{array}
\right.
\end{equation}
where the \emph{dispersal kernel} $J,$ or \emph{mutation kernel} depending on the ecological context,  is a given function and
\[
	(J * n) (t,x):= \int_\R J(x-y) n(t,y) \, dy.
\]
When the term $J*n- n$ is replaced by $\Delta n$, this is the well-known Fisher-KPP equation~\cite{Fisher,fkpp}.

This equation arises naturally in population dynamics to model systems with non--local effects \cite{fife,medlockkot}. In this context, the unknown function $n$ represents a density of individuals at time $t$ and at position $x$. One of the most interesting features of this model, compared to  the classical Fisher-KPP equation, is that it allows for long range dispersal events.

The existence of these events depends critically on the tail of the kernel $J$.  To differentiate between the two regimes that arise, we introduce the following notation. Roughly, we say that the kernel $J$ is {\em thin--tailed}, or {\em exponentially bounded}, if there exists $\lambda > 0$ such that
\begin{equation}\label{def:J_expobounded}
	\int_\R J(h)e^{\la h} \dd h < \infty,
\end{equation}
Otherwise we may say the kernel is {\em fat--tailed}. We make mathematically precise what we call a {\em fat--tailed} kernel in \Cref{hyp:J}.

When the kernel is thin--tailed, solutions exhibit the same behavior as solutions to the Fisher--KPP equation in that travelling wave solutions exist~\cite{aronson_weinberger,diekmann1979}.  This regime can be used to model a biological invasion scenario in which a population invades a homogeneous landscape at constant speed. There is an extensive literature about models similar to \cref{eq:main} investigating the existence and stability of traveling waves solutions, see~\cite{CarChm04,Sch80,WeinbergerNL} along with the work and references contained in the habilitation thesis of Coville~\cite{CovilleHab}. 

On the other hand, super-linear in time propagation phenomena can occur in ecology. A classical example is Reid's paradox of rapid plant migration \cite{clark1998,clark1998trees} that is usually resolved using fat--tailed kernels. 
Indeed, when the kernel is fat--tailed, the solutions of~\ \c	ref{eq:main} do not propagate at constant speed but accelerate with a rate that depends on the thickness of the tail of the kernel $J$~\cite{medlockkot,jimmy2011}. This acceleration phenomenon results from the combination of fat--tailed dispersion events and the Fisher--KPP non--linearity $n(1-n)$ that makes the solution grow almost exponentially when small. In particular, when this cooperation between the kernel and the non--linear term is broken, acceleration can be stopped. Recently, Alfaro and Coville~\cite{AlfCov_P} have proved that traveling wave solutions may exist with fat--tailed kernels when a weak Allee effect is present, that is a non--linearity of the form $n^\beta(1-n)$ with $\beta > 1$, if the tail of the kernel is not too fat.

We shall now be more precise on the type of dispersal kernels that we will consider. We emphasize that the following assumptions on the kernel $J$ hold true throughout this work, even when not explicitly stated: 
\begin{hyp}[Fat--tailed kernel]\label[hyp]{hyp:J}
The kernel $J$ is a symmetric probability density, that is, for all $x \in \R$,
\begin{equation}\label[hyp]{hyp:J_probakernel}
  \int_{\R} J(x)\, dx = 1,
  \qquad  J(x) = J(\vert x \vert)
  \qquad \hbox{ and } \qquad J(x) > 0. 
\end{equation}
The decay of $J$ is encoded in the function 
\begin{equation}\label[hyp]{eq:f}
 f:= - \ln(J), 
\end{equation}
We further assume the following three properties:

{\bf Monotonicity and asymptotic convexity of $J$.} The function $f \in \mathcal{C}^2(\R)$ is strictly increasing on $(0,+\infty)$ and asymptotically concave, that is, there exists $x_{\rm conc}>0$ such that
\begin{equation}\label[hyp]{hyp:J_regconcave}
\begin{cases}
f(x) > f(y), \qquad &\text{ if } x > y \geq 0 \qquad \text{ and }\\
f''( x ) \leq 0 \qquad &\text{  if } x \geq x_{\rm conc}.
\end{cases}
\end{equation}
Without loss of generality, we suppose that $f(0) = 0$, or $J(0) = 1$, since otherwise we may re-scale the equation.  This implies that $J(\R) = (0,1]$. Moreover, $J$ is invertible on $\R^+$, this inverse from $(0,1]$ to $\R^+$ is what we denote $J^{-1}$ in the sequel. Similarly, $f$ is invertible on $\R^+$, this inverse from $\R^+$ to $\R^+$ is what we denote $f^{-1}$ in the sequel.

{\bf Lower bound on the tail of $J$.} The kernel $J$ decays slower than any exponential in the sense that
\begin{equation}\label[hyp]{hyp:J_slowexp}
\limsup_{x\to\infty} \, \frac{x \, f'(x)}{f(x)}<1.
\end{equation}
Roughly speaking, this implies that $f$ grows sub-linearly and that $J'(x) = o_{x\to\infty}(J(x))$.

{\bf Upper bound on the tail of $J$.} The tail of $J$ is thinner than $\vert x \vert^{-1}$, in the sense that
\begin{equation}\label[hyp]{hyp:J_poly}
\liminf_{x  \to \infty} x\, f'( x ) > 1.
\end{equation}
For ease of notation and since it will play a role in our analysis, we define $\mu:= \liminf_{x\to\infty} x\, f'(x)$.
\end{hyp}

The main examples of kernels $J$ that satisfy \Cref{hyp:J} are  either sub--exponential kernels where $f(x)= (1 + |x|^2)^{\alpha/2}$ with $\alpha<1$ or polynomial kernels where $f(x) = \alpha  \ln (1+ \vert x \vert^2)/2$ with $\alpha > 0$.
Our technical assumptions~\eqref{hyp:J_regconcave}--\eqref{hyp:J_poly} do not cover borderline kernels such as $f(x) = \vert x \vert/ \ln( 1 + \vert x \vert^2 )$ that were considered by Garnier~\cite{jimmy2011}. Moreover, we restrict our focus to the effects of the tails of $J$ on the rate of propagation. As a consequence, we do not include potential singularities at the origin, which is the case for a fractional Laplacian operator, for example. We expect however that our results also hold for these cases.

Garnier~\cite{jimmy2011} proved that the acceleration propagation of the solution of~\eqref{eq:main} can be measured by tracking the level sets $E_\la(t):=\{x \in \R: n(t,x)=\la \}$ of the solution $n$, where $\lambda\in(0,1)$.  Under the fat--tailed kernel hypothesis, these level sets move super-linearly in time. More precisely,  Garnier~\cite{jimmy2011} proved that there exists a constant $\rho>1$ such that for any $\lambda\in(0,1)$ and any $\e>0$, any element $x_\lambda$ of the level set $E_\lambda$ satisfies for all $\e > 0$ and $t$ large enough:
 \begin{equation}\label{eq:levelsets_jimmy}
	J^{-1}\left(e^{-(1-\e)t}\right)
		\leq |x_\la(t)| 
		\leq J^{-1}\left(e^{-\rho t}\right).
\end{equation} 
The propagation problem has been recently considered with non--symmetric kernels in \cite{Finkel17} and in the multi-dimensional case in \cite{Finkel16}.  There Finkelshtein, Kondratiev, and Tkachov present a technical argument improving the precision in Garnier's bounds. As will be made clear below, one goal of this paper is to connect these results to an underlying Hamilton-Jacobi equation, giving a new interpretation to the speed of propagation.  In doing so, we aim to provide a simpler proof of propagation that is, in some regimes, more precise than those above.

The approach we use comes from the seminal paper of Evans and Souganidis~\cite{EvansSouganidis}, where the authors applied the long time/long range limit to the Fisher-KPP equation and showed convergence, after applying the Hopf-Cole transform, to a Hamilton--Jacobi equation.  Further, they show that the solution to this Hamilton--Jacobi equation determines the propagation rate of the original Fisher-KPP equation.
In the first part of our work, we first describe the scaling corresponding to the long time/long range limit for~\cref{eq:main}. Then we show convergence to a Hamilton--Jacobi equation, which gives information on the spreading properties of the solution to the original equation~\cref{eq:main}.

The same approach has been used extensively to understand propagation in various physical systems as well as more general qualitative behavior of solutions of parabolic equations. In the context of adaptative dynamics, Diekmann, Jabin, Mischler, and Perthame in~\cite{diekmann2005} have provided an example of this approach. They have derived, by a limiting procedure, a Hamilton--Jacobi equation from a mutation--selection equation with small mutations. 

The literature on this topic is enormous, but we mention a few closely related works.
Perthame and Souganidis studied~\cref{eq:main} with thin--tailed kernels in~\cite{PerSou05}. Barles, Mirrahimi and Perthame focused on Dirac concentration in integro-differential equations with local and non--local non--linearities in \cite{barles2009}. Recently, Mirrahimi and M\'el\'eard  extended this approach to a fractional diffusion~\cite{SMSM}. The second part of our work is inspired by and closely follows this last work, extending their technique to the case with a general dispersal kernel $J$, the main difficulties coming from the fact that, in contrast to the fractional Laplacian, the kernel $J$ is not explicit and has no natural scaling. As a consequence, the scalings we use are not always easy to read but we give some heuristics that make them appear naturally. Moreover, our scalings allow to characterize the small mutation regime when the mutation kernel is fat--tailed.

In the present work, we only focus on the local Fisher--KPP nonlinearity $n(1-n)$ but our results can be generalized to a non--local non--linearity of the form $n(r-\int_\R n(x)dx)$, exactly as in \cite{SMSM}. This non--local term arises also naturally in the context of mutation--selection model or structured population models. In this context, $x$ denotes a quantitative trait and $n(t,\cdot)$ describes the distribution of this trait inside the population. Thus, the parameter $r$ describes the fitness of the population and the integral term $\int_\R n(t,x)dx$ is a mean--competition term. The model~\cref{eq:main} with this non--local nonlinearity can be derived rigorously from an individual--based model where mutations are described by a fat--tailed kernel without jump (see for instance~\cite{Bae07,GurNis75}). In general, the growth rate $r$ depends also on the trait parameter $x$~\cite{Jou12}. However, this general form induces more technical difficulties that we do not tackle in this paper.

\subsection*{The propagation regime}
In order to capture the accelerated propagation phenomenon that occurs with fat--tailed kernels, we look at the behavior of $n$ in the long time/long range limit.  Indeed, we first rescale the time $t \mapsto t/\e$ by a small parameter $\e$ and then we need to find an accurate rescaling in space that captures the propagation regime. We thus look for a space rescaling of the form $x\mapsto \psi_\e(x)$. The seminal paper \cite{EvansSouganidis} used the hyperbolic scaling $( t/\e, x/\e)$  for the asymptotic study of the Fisher-KPP equation.
The precise shape of $\psi_\e$ will be given below, but we first start with heuristic explanation of this expression.

Using the results of Garnier \cref{eq:levelsets_jimmy}, it is reasonable to say that the position $x$ of any level sets satisfies
\begin{equation*}
J(x) \sim e^{-t}.
\end{equation*}
Our aim is to find a rescaling $\psi_\e$ that follows the level sets in the long time rescaling $t/\e$. So  we want
\begin{equation*}
J(\psi_\e(x)) \sim e^{- \frac{t}{\e}} \sim \left( e^{-t}\right)^{\frac{1}{\e}} \sim J(x)^{\frac{1}{\e}}.
\end{equation*}
As a consequence, for any $\varepsilon>0$, it is natural to set 
\begin{equation}\label{eq:rescaling_psi}
\psi_\e(x)= \sign(x) J^{-1}\left(J( x )^\frac{1}{\e}\right) \ \hbox{ for all } x\in\R.
\end{equation}

The rescaling $\psi_\varepsilon$ transforms functions that looks like $J$ into functions that looks like $ \exp\left\{\left(\ln J(x)\right)/\epsilon\right\}.$ Indeed, notice that $J(\psi_\epsilon(x)) = J(x)^{1/\epsilon} = \exp\left\{\left(\ln J(x)\right)/\epsilon\right\}.$   Since the solution $n$ of the Cauchy problem~\cref{eq:main} is expected to behave like $e^t J$ for large $x,$ we can heuristically say that the rescaled function $n_\epsilon = n(t/\epsilon,\psi_\epsilon)$ should behave like $e^{t/\e}J(\psi_\epsilon(x))= \exp\left\{(t+\ln J(x))/\e\right\}\sim \exp\left\{(\ln n(t,x))/\epsilon\right\}.$ The last expression is the logarithmic Hopf--Cole transformation of $n(t,x).$ Our rescaling is thus compatible with this transformation.

The scaling $\psi_\e$ can also be rewritten in terms of the function $f$ introduced in \cref{hyp:J}.  Indeed, $$\psi_\e(x)= \sign(x) f^{-1}\left(f(x)/\e\right).$$
We derive a precise formula for this scaling for our two main examples: the sub--exponential kernels and the polynomial kernels.
\begin{example}[Sub--exponential kernels]
Consider $f(x):=   (1 + |x|^2)^{\alpha/2} - 1$ with $\alpha \in (0,1)$.
Then 
$$\psi_\e(x) =  \sign(x)  \left[ \left( 1+ \frac{1}{\e} \left[ \left( 1 + \vert x \vert^2 \right)^\frac{\alpha}{2} -1 \right] \right)^\frac{2}{\alpha} - 1 \right]^\frac12.  $$ 
Observe that $\psi_\e(x) \sim \e^{- \frac{1}{\alpha}}  x $ when $\vert x \vert \to + \infty$ at fixed $\e$ and $\psi_\e(x) \sim \e^{- \frac{1}{\alpha}} \sign(x) $ $ \left[ \left( 1 + \vert x \vert^2 \right)^\frac{\alpha}{2} -1 \right]^\frac{1}{\alpha}$ when $\eps \to 0$ and $x \neq 0$.
\end{example}

\begin{example}[Polynomial kernels]
Consider $f(x):= (1+\alpha)\ln\left( (1 + |x|^{2})^{\frac12} \right)$ for $\alpha > 0$.  In this case, the scaling becomes 
\[
	\psi_\e(x) =  \sign(x) \sqrt{ (1 + x^2)^{1/\e} - 1}.
\]
One can observe that, for any fixed $\e$ and $\alpha$, $\psi_\e(x) \sim  \sign(x)  \vert x \vert^\frac{1}{\e}$ as $\vert x \vert \to + \infty$.   In addition, $\psi_\e(x) \sim  \sign(x) \left(1 + \vert x \vert^{2}\right)^\frac{1}{2\e}$, when $\e \to 0$ at fixed $x \neq 0$.  We point out that when $\alpha\in(0,2),$ the kernel $J$ decays at the same rate as the kernel in the definition of the fractional Laplacian $(-\Delta)^{\alpha/2}$ as $|x|\to\infty.$  This suggests that the behaviour of the solution of~\cref{eq:nep_prop} with $f$ as above and $\alpha\in(0,2)$ and the behaviour of the solution of~\cref{eq:main} $(-\Delta)^{\alpha/2}n$ in the place of $J*n - n$ are the same.  We verify this below. We also point out that  in the limit $\vert x \vert \to + \infty$, $\psi_\e(x) \sim \sign(x)\vert x \vert^\frac{1}{\e}$, which is the rescaling chosen in \cite{SMSM} for the fractional Laplacian.  
\end{example}

As far as the initial data is concerned, we assume without lost of generality that there exists two positive constants $\underline{C}$ and $\overline{C}$ such that $\underline{C}<1<\overline{C}$ and
\begin{equation}\label[hyp]{cond:n0_likeJ}
\underline{C} J  \leq n^0  \leq \overline{C} J.
\end{equation}
 Moreover, we assume that the initial data is symmetric, and since the kernel $J$ is also assumed to be symmetric (see \cref{hyp:J_probakernel}), the solution $n$ thus remains symmetric for all times.
\begin{remark}\label{rem:J}
  The lower bound in assumption~\eqref{cond:n0_likeJ} is not restrictive, though it allows us to avoid discussion of a boundary layer at $t=0$. Indeed, assuming \Cref{hyp:J}, any solution of~\cref{eq:main} starting with initial data that decays faster than $J$ at infinity satisfies~\cref{cond:n0_likeJ} after at most time $1$. More precisely, for any $n^0$ decaying faster than $J$, there exist constants $\underline{C}$ and $\overline{C}$ such that $\underline{C} J(x)  \leq n(1,x)  \leq \overline{C} J(x)$ for any $x\in\R$ see, e.g.,~\cite[Section~4.2]{jimmy2011}). After translating in time, our argument applies with initial data $n(1,x)$.  From the uniqueness of solutions of the Cauchy problem~\cref{eq:main}, our conclusions hold for $n$ as well. 
If, on the other hand, the upper bound in~\cref{cond:n0_likeJ} does not hold, i.e.~$n^0$ decays slower than $J$, then we expect different behavior.  Indeed, by analogy with~\cite{Alfaro, HamelRoques, Henderson_Acceleration}, we expect faster propagation depending only on the rate of decay of the initial data.

In view of the above, the assumption~\eqref{cond:n0_likeJ} is quite general for the regimes that we wish to understand.
\end{remark}

Let us now rescale time and space as follows: $t\mapsto t/\e$ and $x\mapsto\psi_\e(x)$ and define the solution $\nep$ in the new variables: $n_\e(t,x) = n(t/\e,\psi_\e(x))$ where $n$ solves~\eqref{eq:main} with initial condition $n^0$ satisfying~\eqref{cond:n0_likeJ}. Plugging this quantity into~\eqref{eq:main}, we obtain the following equation:
\begin{equation}\label{eq:nep_prop}
\begin{cases}
\e \partial_t n_\e = \displaystyle\int_{-\infty}^\infty J(h) \left[ n_\e \left(t,\psi_\e^{-1}\left( \psi_\e( x ) - h  \right) \right) - n_\e \right] dh \; + \; \nep \left( 1 - n_\e \right)& \!\!\!\hbox{ in } (0,\infty)\times\R, \medskip\\
\nep(0,x)=n^0(\psi_\e(|x|)),& \!\!\!\! \text{ for } x\in\R.
\end{cases}
\end{equation}
We know from \cite{jimmy2011} that the solutions of~\cref{eq:nep_prop} will propagate and converge to one as $t\to\infty$. In the large scale limit $\epsilon\to0$ with our change of variables, we expect this propagation to be transformed into dynamics of an interface moving with time. To capture this phenomenon, we use the logarithmic Hopf--Cole transform~\cite{EvansSouganidis,Freidlin1985} as follows:
\begin{equation}\label{eq:HopfCole_prop}
u_{\e}:= - \e \ln n_\e.
\end{equation}
Notice that this is equivalent to $\nep=\exp\left(-\frac{\uep}{\e}\right)$. 
Then, the function $\uep$ solves:
\begin{equation}\label{eq:uep_prop}
\begin{cases}
\!\partial_t \uep + 1
		 \! = \! \displaystyle\int_{-\infty}^\infty \!\! J(h) \left[ 1 - e^{ - \frac{1}{\e} \left(  \uep \left(t, \psi_\e^{-1}\left( \psi_\e( x ) - h  \right) \right) - u_\e \left( t,x \right) \right)}   \right] dh \;   + n_\e, \ \! \!\!\hbox{ in } (0,\infty)\times\R\medskip\\
\uep(0,x)= -\e \ln \big( n^0(\psi_\e(x)) \big), \ x\in\R.
\end{cases}
\end{equation}

Note that assumptions \eqref{cond:n0_likeJ} imply that $\uep(0,\cdot) \to f$  uniformly in $\R$ as $\epsilon\to0$. Our aim is to compute the limit $\e \to 0$ of $u_\eps$ and then deduce the behavior of $n_\eps$. The result is the following.

\begin{theorem}\label[thm]{thm:u_conv}
Let $\uep$ be the solution of~\cref{eq:uep_prop} with initial condition satisfying~\cref{cond:n0_likeJ}. If the kernel $J$ satisfies \cref{hyp:J}, then as $\e \to 0$, the sequence $u_\eps$ converges locally uniformly on $(0,\infty) \times \R$ to
\begin{equation*}
u(t,x) := \max\{f(x) - t,0\}.
\end{equation*}
\end{theorem}
From this convergence result, we may deduce the asymptotics of $\nep$.
\begin{theorem}\label[thm]{thm:n_conv}
Let $\nep$ be the solution of \cref{eq:nep_prop} with the initial data satisfying~\cref{cond:n0_likeJ}. If the kernel $J$ satisfies \cref{hyp:J}, then
\begin{itemize}
\item[(a)] uniformly on compact subsets of $\{ u > 0
\}$, \quad $$\lim_{\e \to 0} n_{\e} = 0;$$
\item[(b)] for every  compact subset $K\subset\Int \left( \left\lbrace u(t,x) = 0\right\rbrace \right)$, $$\lim_{\e \to 0} n_\e (t,x) = 1,$$
where the limit is uniform in $K$.
\end{itemize}
\end{theorem}

Since $f$ is a continuous and increasing function of \!\! $\vert x \vert$, the boundary of $\left\lbrace u(t,x) = 0\right\rbrace$ is given by $\vert x \vert = f^{-1}(t)$.  Hence, as $\e\to 0$, $n_\e \sim 1$ if and only if $|x| < f^{-1}(t)$.   Since $n(t,x) \sim n_\e(1,\psi_\e^{-1}(x))$ with $\e= 1/t$, then as $t \to \infty$ we see that $n(t,x) \sim 1$ if and only if $|x| < f^{-1}(t)$.  As such, \Cref{thm:u_conv} and \Cref{thm:n_conv} imply that the location of the front of $n$ is $\sim f^{-1}(t)$.

 Let us apply our two main results \Cref{thm:u_conv} and \Cref{thm:n_conv} to our basic examples.
\begin{example}
When $f$ is a sub--exponential kernel of the form $f(x) = (1 + |x|^2)^{\alpha/2}$ with $\alpha \in (0,1)$, we see that the front is located at $\sim t^{1/\alpha}$.  In the thin-tailed limit $\alpha \to 1$ see recover constant speed propagation.

On the other hand, when $f$ is a polynomial kernel of the form $f(x)= (1+\alpha)\ln(1+ \vert x \vert^2)/2$,   with $\alpha > 0$, we see that the front is located at  $\sim e^{t/(1+\alpha)}$.
\end{example}

In \cref{thm:u_conv} the dispersion kernel $J$ disappears when we pass to the limit $\e \to 0$, in the sense that the constrained Hamilton--Jacobi equation satisfied by the limit function $u$ is simply $\min\left\{ \partial_t u +1 , u \right\}=0$, in which $J$ is absent. Solutions to this equation are given by $\max\{u_0(\cdot) - t, 0\},$ so that the effect of the kernel is felt only through the initial data.  Without the \cref{cond:n0_likeJ}, a boundary layer at $t=0$ would develop during the limit $\e \to 0$.  See \cref{rem:J}.

 This is quite different from the case of a thin tailed kernel, for which the Hamiltonian would typically contain a term of the form $\hat J - 1$~\cite{diekmann2005,PerSou05,barles2009}. This is explained by our rescaling $\psi_\e$ which focuses on the behavior at infinity of the solution, and thus mainly ignores the precise dynamics of the dispersion. This phenomenon was already observed in \cite{SMSM} for the fractional Laplacian. Despite this, \Cref{thm:n_conv} states that our rescaling is sharp enough to capture the interface at infinity.

One way to understand intuitively why the kernel disappears in the limiting equation is to investigate the integral term in \cref{eq:uep_prop}. Due to the fat--tailed assumption in \Cref{hyp:J}, the quantity $\psi_\e^{-1}\left( \psi_\e (x) + h \right)- \psi_\e^{-1} (\psi_\e (x))$ is likely to go to zero faster than $\e$. Hence, the integral disappears in the limit.
While this is formally clear, it is difficult to make this intuition rigorous.

We comment momentarily on the method of proof.  We construct explicit sub- and super-solutions of $u$ using the kernel $J$ and the general solution to the logistic equation.  While the most natural thing to do would be to use half relaxed limits along with the limiting Hamilton-Jacobi equation (see \cite{barles1990comparison}), the non--locality of the kernel makes this very difficult because the non-local term in the equation ``sees'' all of $\R$ but the half-relaxed limits only provide convergence locally.  Thus, as in \cite{SMSM}, we construct sharp sub- and super--solutions of~\cref{eq:uep_prop} to conclude. The construction of these sub- and super--solutions also provides sharp sub- and super--solutions for equation~\cref{eq:main}. We point out that \cref{thm:n_conv} improves the existing bounds in~\cite{jimmy2011}.

To illustrate the results of \Cref{thm:u_conv}, we provide the results of some numerical simulations in \Cref{fig:Shape} for four choices of kernels $J$: a Gaussian kernel for which linear spreading is expected \cite{coville2005}, two sub-exponential kernels $J \sim \exp(- \vert \cdot \vert^{1/2})$, $J \sim \exp(- \vert \cdot \vert^{3/4})$ and a polynomial one $J \sim \left( 1 + \vert \cdot \vert^5 \right)^{-1}$.

\begin{figure}[!h]
\begin{center}
\includegraphics[width = .45\linewidth]{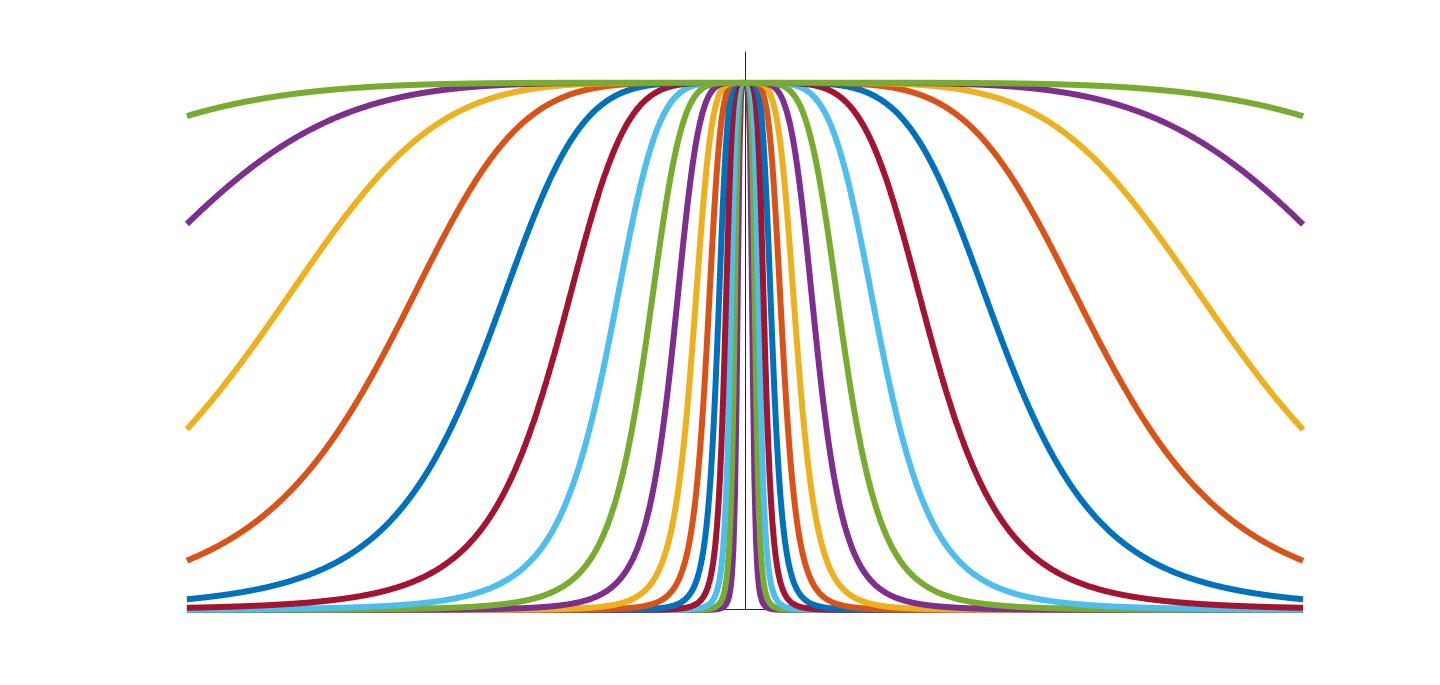}
\includegraphics[width = .45\linewidth]{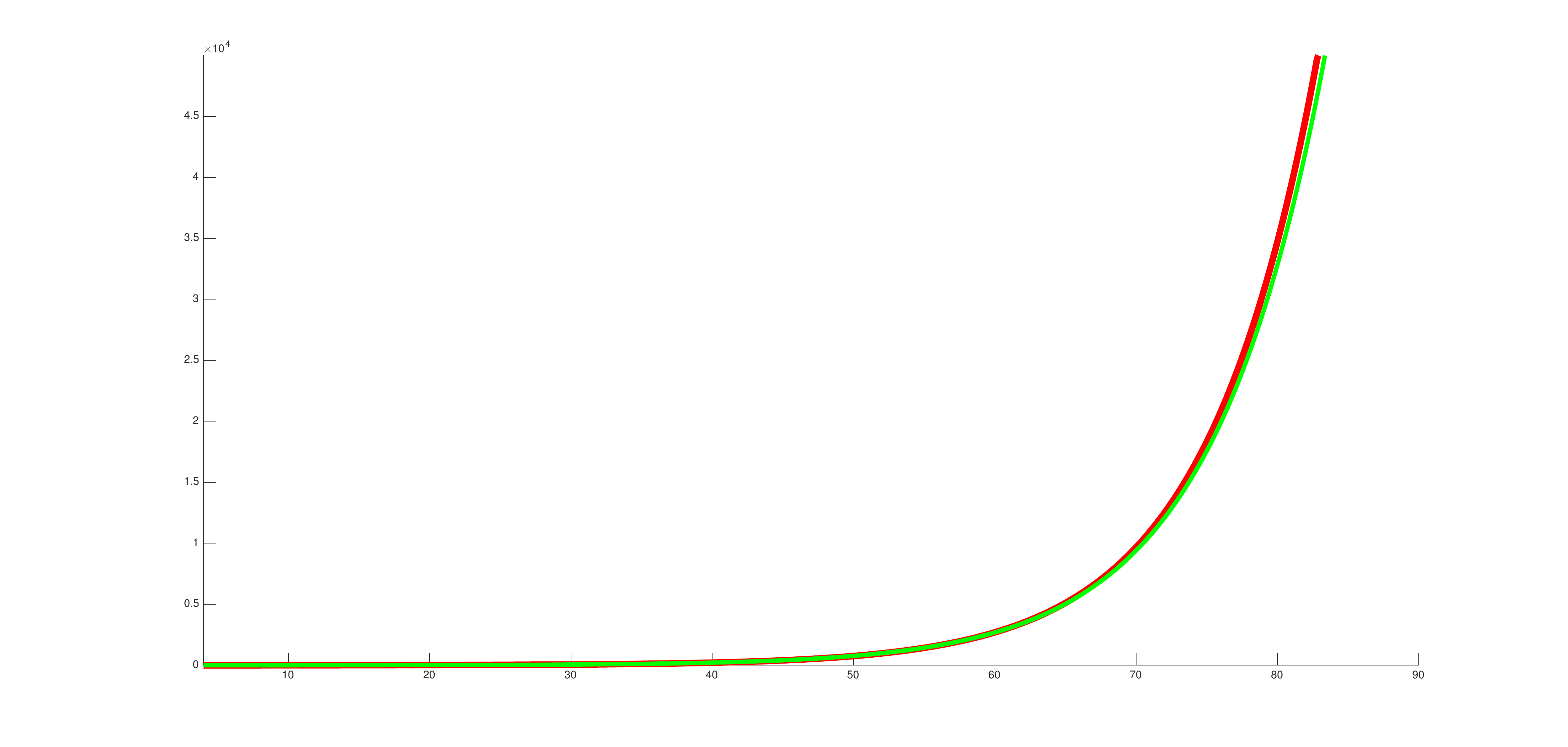}
\\
\includegraphics[width = .45\linewidth]{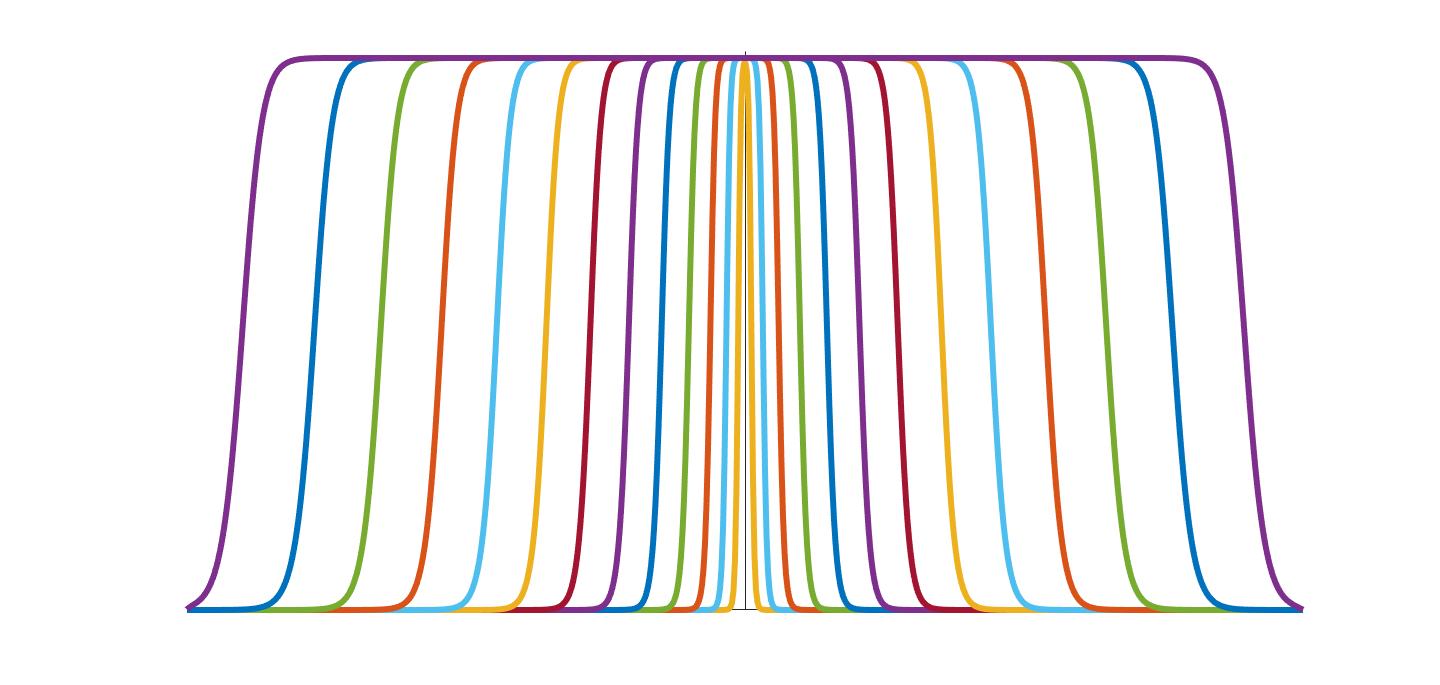}
\includegraphics[width = .45\linewidth]{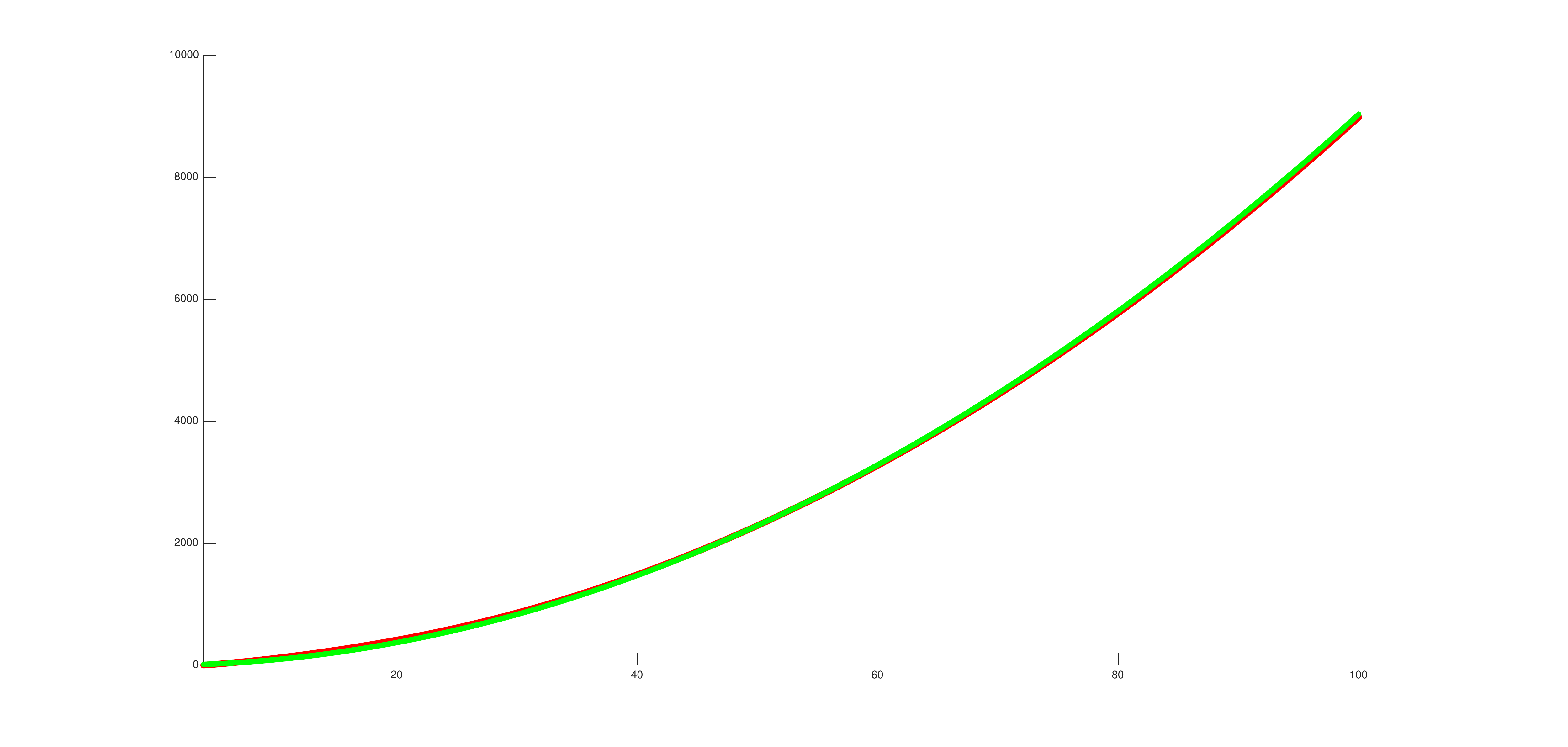}
\\
\includegraphics[width = .45\linewidth]{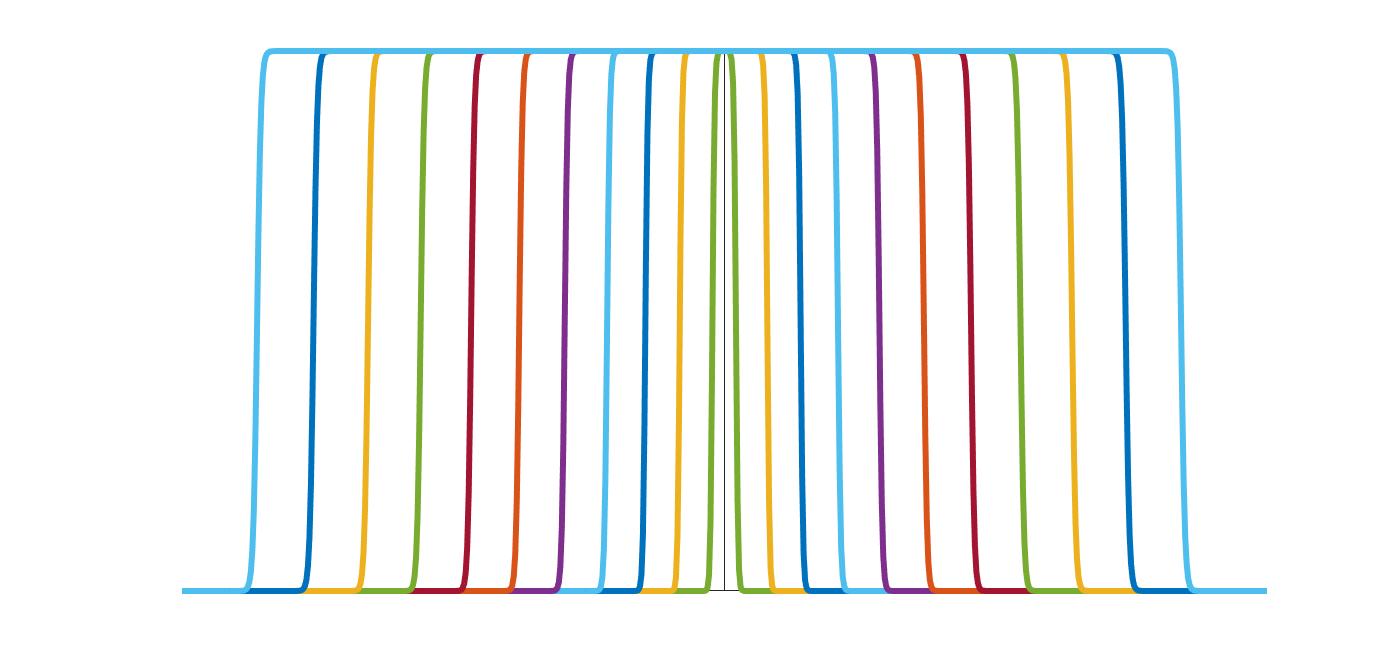}
\includegraphics[width = .45\linewidth]{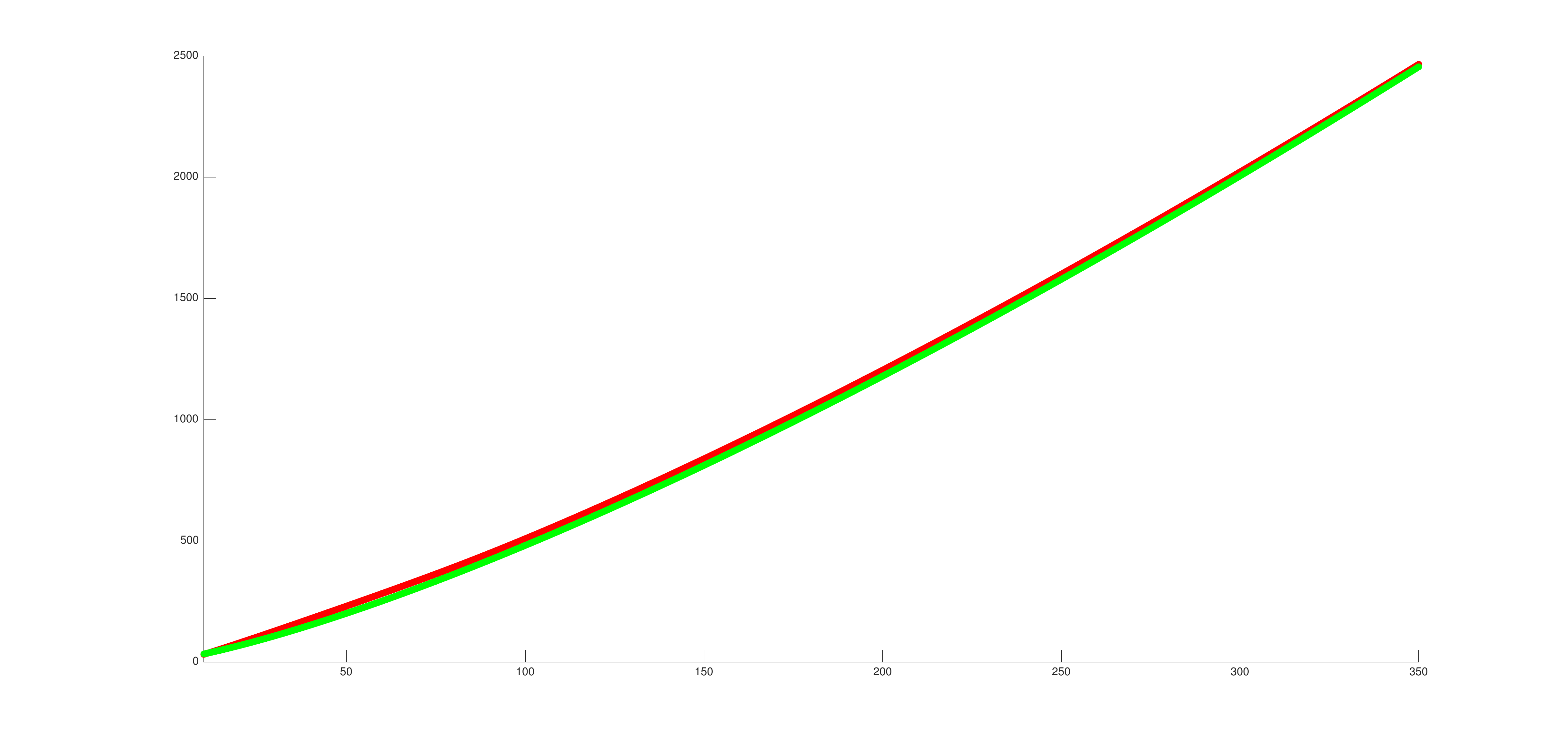}
\\
\includegraphics[width = .45\linewidth]{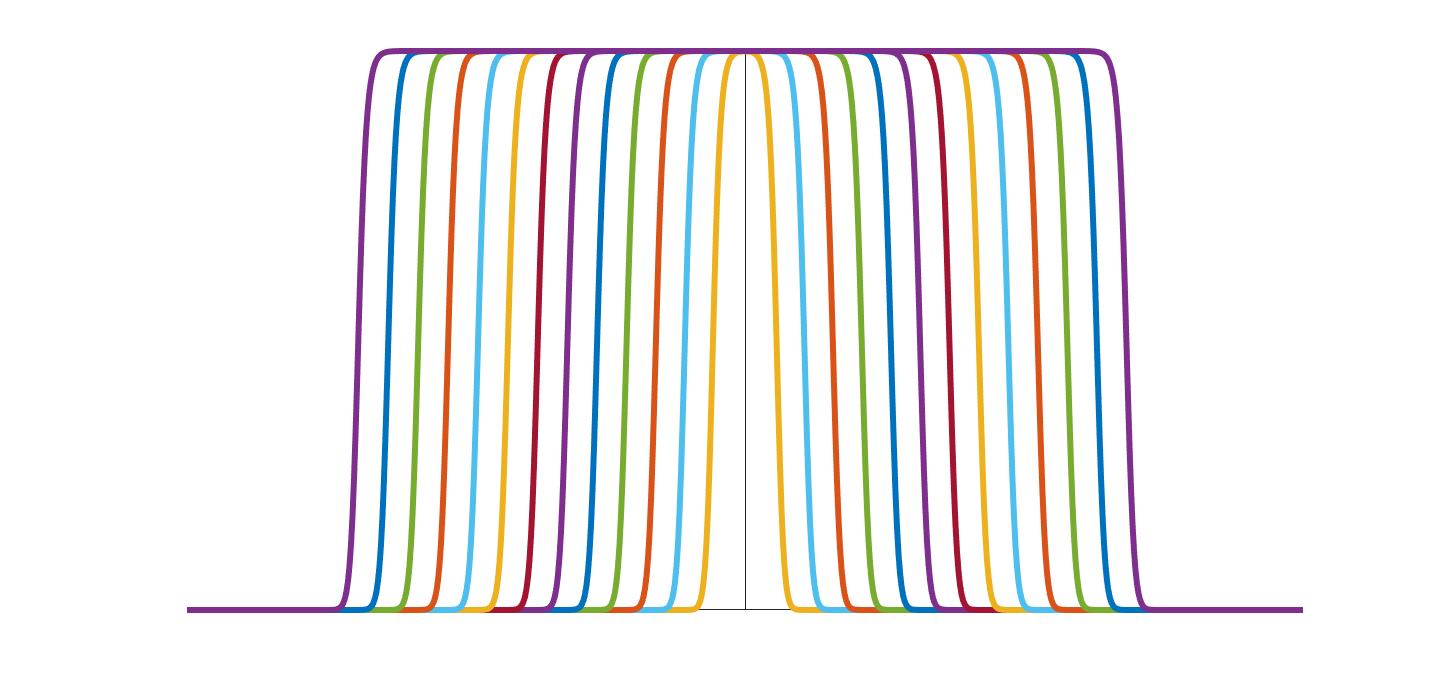}
\includegraphics[width = .45\linewidth]{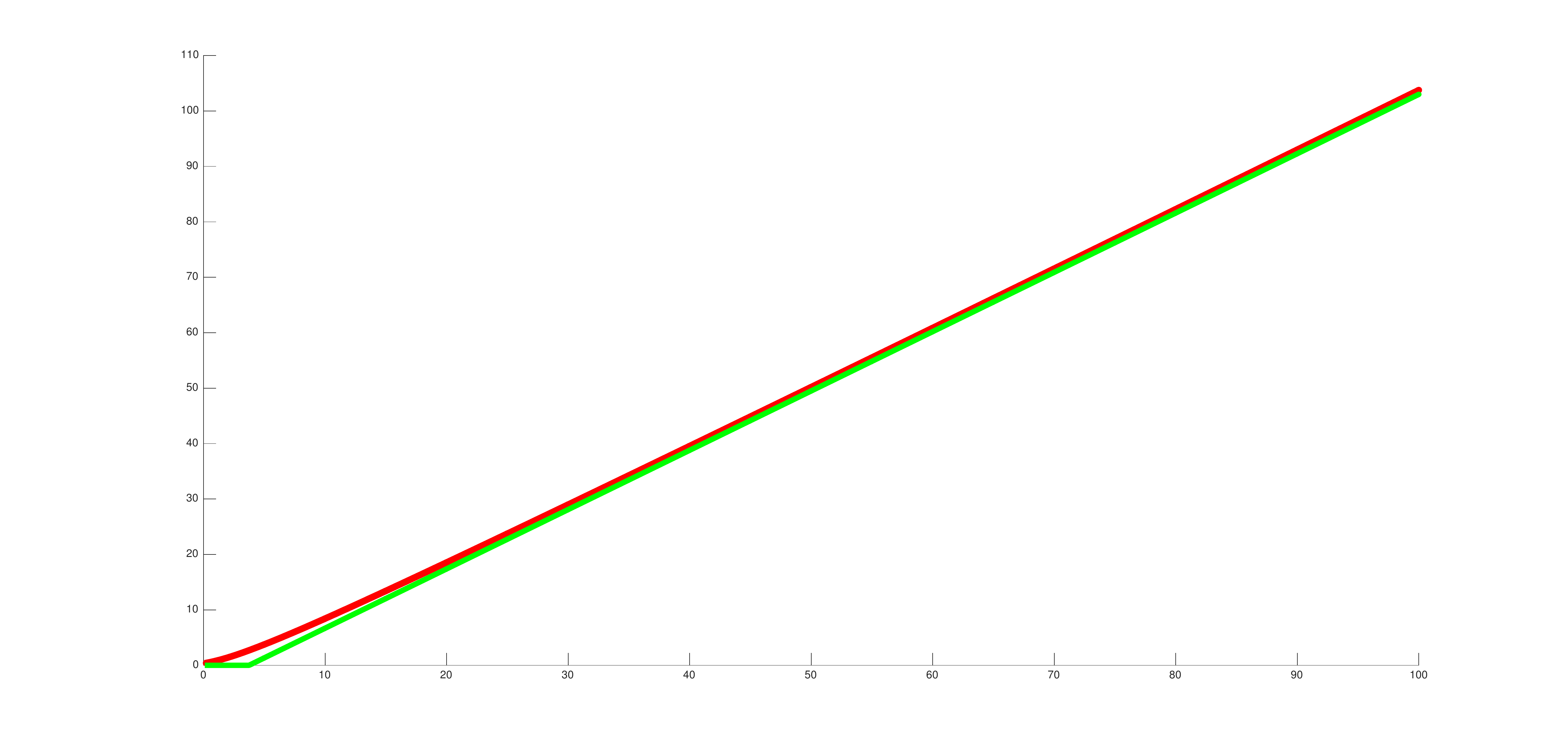}

\caption{Numerical simulations of the Cauchy problem \eqref{eq:main}, for various kernels $J$. The first line corresponds to $J \sim \vert x \vert^{-5}$, the second line to $J \sim \exp\left( - \vert x \vert^{1/2} \right)$, the third line to $J \sim \exp\left( - \vert x \vert^{3/4} \right)$ and the last line to a Gaussian kernel $J \sim \exp\left( - \vert x \vert^{2} \right)$. 
For each kernel, we present in the left column the evolution of the solution by plotting it on the same figure for various successive (linearly chosen) values of time $t$.  To quantify this and recover and illustrate \Cref{thm:n_conv}, we present in the right column the time evolution of the level set $\{x \in \R: n(t,x)=1/2\}$ for each kernel. The red bold curve is the numerical simulation, starting from an initial condition of the form $J$. The green curve is the expected asymptotic rate of expansion predicted by \Cref{thm:u_conv}, that is $J^{-1}(e^{-t})$, except for the Gaussian kernel, in which case it is a line.    Each kernel is successively more thin--tailed, yielding an interpolation between the obvious acceleration for the first kernel and the linear propagation for the last kernel.}
\label{fig:Shape}
\end{center}
\end{figure}

\subsection*{The small mutations limit}

Our main equation~\cref{eq:main} also arises naturally in the context of population genetics, to capture the effect of genetic mutations \cite{barles2009,PerBar08}. Under this perspective, the variable $x$ now corresponds to a phenotypic trait and the convolution term describes the mutation process during which an individual with trait $x$ can give birth to an individual with trait $x+h$ with probability $J(h)$.

We are interested in a situation where large--effect mutations, while still uncommon, are {\em relatively} frequent. This is exactly what is encoded in a mutation kernel with fat--tails. The aim of this section is to understand the effect of these large mutation events on the adaptive dynamics when the mean effect of mutations is small. This regime of small mean--effect of mutations will be referred as the small mutation regime. Note that even in the small mutation regime, mutation events with a large effect can occur. The main difficulty is to identify the appropriate scaling of this small mutation regime when the mutation kernel has fat tails.

To work in the small mutation regime, we introduce a small parameter $\epsilon$, such that $\epsilon^{-1}$ typically represents the time-scale on which these mutations accumulate. This time scale being given, one needs to scale the size of the mutations in a relevant way to capture the expected (non-trivial) dynamics.

In the case of a thin--tailed mutation kernel, the small mutation regime corresponds to mutation kernel with small variance of order $\epsilon^2.$ Thus, it is natural to rewrite the mutation kernel $J(h)$ as $J_\epsilon(h):=J(h/\epsilon)/\epsilon$ where now $J$ is of variance $\sigma^2$.  Precisely the variance of $J_\e$ is equal to $\sigma^2 \e^2$. With such transformation, the jump $x \mapsto x+h$ with probability $J(h)$ is replaced by the jump $x \mapsto x+\epsilon h$ with the same probability $J(h)$. In this case, the asymptotic behavior of the population is described by a Hamilton--Jacobi equation~\cite{barles2009,diekmann2005,PerBar08}.

In the fat--tailed setting, however, the large mutation events modify the dynamics.  As a consequence, it is necessary to rescale the jump size non-linearly to take this into account as the scaling above does not contract the kernel enough. Indeed, if the size of jumps is of order $\e$ it does not go to zero fast enough to be comparable with the contraction of the kernel. Inspired by our first part about propagation, we replace the jumps $x\mapsto x + h$ by $x \mapsto x + \psi_\epsilon^{-1}(h)$ with the same probability, where
$$\psi_\eps(h) = \sign(h)f^{-1}\left( \e f(h)	 \right)\,.$$ 
Our jump size scaling procedure is actually equivalent to rewrite the mutation kernel $J$. The mutation kernel $J$ is transformed into $J_\epsilon$ defined by
\begin{equation}\label{eq:rscl_J}
J(h)=e^{-f(h)} \longmapsto J_\varepsilon (h)= J(\psi_\e)\psi_\e'(x)=\dfrac{1}{\e} \frac{f'(h)}{f'\left(\psi_\e^{-1}(h)\right)} J(h)^{\frac1\e}.
\end{equation}
Thus, with fat--tailed mutation kernel, the small mutation regime corresponds to when the following quantity is rescaled by $\epsilon^2$:

$$\int_\R f(x)^2 J_\epsilon(x)dx =\epsilon^2 \left( \ds \int_\R f(x)^2 J(x)dx \right)$$
Observe that in the case of the small mutation regime of a thin tailed kernel, this quantity would be exactly the variance of $J_\eps$.
From the formula~\cref{eq:rscl_J}, we can observe that $J_\epsilon$ is a contraction of $J$.

Due to the small size of the mutations, their effect can only be seen after many mutations accumulate.  Hence, we want to capture the long time behavior of the population, or, equivalently, the setting where the rate of mutation is large.  This suggests that we rescale the time by the parameter $\epsilon$ as  $t\mapsto t/\epsilon$. Under this rescaling and the rewriting of the mutation~\cref{eq:rscl_J}, equation~\cref{eq:main} becomes:
\begin{equation}\label{eq:nep_mut}
\e\partial_t \nep(t,x)=\int_\R \left(\nep(t,x-\psi_\e^{-1}(h)) -\nep (t,x)\right) J(|h|) dh  +\nep(1-\nep).
\end{equation}
As in the propagation regime, the scaled size of jumps goes to $0$ as $\epsilon\to0.$ So we expect the solution $\nep$ to concentrate. In order to capture this concentration phenomenon, we perform the logarithmic Hopf--Cole transformation, $u_\epsilon:=-\epsilon\ln(\nep)$ and $u_\e$ satisfies the following equation:
\begin{equation}\label{eq:uep_mut}
\begin{cases}
   \partial_t \uep(t,x) +  \int_\R e^{-\frac{1}{\e}\left(\uep\left( t, x - \psi_\eps^{-1}(h)  \right)-\uep(t,x) \right)} J(h) \dd h  = \nep(t,x), \quad &\hbox{ on } \ (0,+\infty)\times\R \\
\uep(t,\cdot)=\uep^0.
\end{cases}
\end{equation}

Before stating our main results, we need the following additional technical assumption on the derivative of $f = -\ln(J)$ at $x=0.$ 
\begin{hyp}\label[hyp]{hyp:Jmut}
Assume that $f \in C^2(0,\infty)$ and that $f$ satisfies $\lim_{x\searrow 0} f(x)/x \in (0,\infty)$. We abuse notation by denoting it $f'(0)$.  Additionally, $f$ satisfies all the assumptions of \Cref{hyp:J} except for regularity at zero.  Namely, $f\in C^0(\R)$ but $f$ is not $C^2$ at $0$.
\end{hyp}

We also require additional assumptions on the initial data $u_\epsilon^0 = -\epsilon\ln(\nep^0).$  We assume that $u_\e^0$ is a positive sequence of Lipschitz continuous functions which converges locally uniformly to $u^0$ as $\e \to 0$ and there exists $A\in(0,1-1/\mu)$ where, we recall, $\mu:=\liminf_{x\to+\infty}{|x|f'(|x|)}>1$ such that for all $x, h \in \R$:
\begin{equation}\label[hyp]{cond:uep0_mutation}
u_\e^0(x+h)-u_\e^0(x)\geq -Af(h).
\end{equation}
Note that $(u_\e^0)_\e$ is thus uniformly locally bounded and $\nep^0=\exp(-u^0_\e/\e)$ satisfies, for all $x \in \R$ and $\e>0$,
\begin{equation}\label[hyp]{cond:nep0_boundmut}
0 \leq \nep^0(x) \leq 1.
\end{equation}
From the maximum principle, we have that $0 < n_\e(t,x) < 1$ for all $x \in \R$ and $t>0$. Moreover, the property~\cref{cond:uep0_mutation} propagates for any positive time -- see the following lemma.

\begin{lemma}\label[lem]{lem:estim_uep}
Let $f$ satisfy \Cref{hyp:Jmut}. Then any solution $u_\e$ of~\cref{eq:uep_mut} starting with initial condition $u^0_\epsilon$ satisfying~\cref{cond:uep0_mutation}, satisfies the following properties:
\begin{enumerate}
\item the sequence $(\uep)_\e$ is locally uniformly bounded. In particular, there exists $r > 0$ such that, for all $x \in \R$ and $t\in (0,\infty)$,

\begin{equation*}
- rt \leq u_\e(t,x) - u_\e^0(x) \leq t;
\end{equation*}
\item for all $t \geq 0$ and all $x,h \in \R$, we have:
\begin{equation}\label{ineq:estim_ueph}
u_\e(t,x+h)-u_\e(t,x)\geq -A f( |h| ). 
\end{equation}
In particular $\uep$ is uniformly Lipschitz with respect to $x$ with the bound
\begin{equation}\label{ineq:graduep}
\left\| \partial_x  \uep\right\|  _{L^\infty((0,+\infty) \times\R )} \leq Af'(0).
\end{equation}
\end{enumerate}
\end{lemma}

The local uniform estimates of~\Cref{lem:estim_uep} allows us to define the following upper- and lower--half--relaxed limits of $u_\e$ by the following formulas:
\begin{equation}\label{eq:half_relaxed}
	\underline u(t,x):= \liminf_{\underset{(s,y)\to (t,x)}{\e\to 0}} \ u_\e(s,y)
	\qquad \text{ and }\qquad
	\overline u(t,x):= \limsup_{\underset{(s,y)\to (t,x)}{\e\to 0}} \ u_\e(s,y)\,.
\end{equation}
From the properties of half-relaxed limits, the estimates~\cref{ineq:estim_ueph} and~\cref{ineq:graduep} hold true for the functions $\underline u$ and $\overline u$.  In addition, it is apparent that $\underline u \leq \overline u$ in $(0,\infty)\times\R$ by construction. With this sub- and super--solution in hands, we can state our main results on the asymptotics of $u_\e$ and $n_\e$ as $\epsilon\to0.$

\begin{theorem}\label[thm]{theo:mut}
Let $f$ satisfy \Cref{hyp:Jmut} and let $\uep$ be the solution of~\cref{eq:uep_mut} starting with initial condition $u^0_\epsilon$ satisfying~\cref{cond:uep0_mutation}. Then as $\e \to 0$ we have the following:
\begin{itemize}
\item[i)] the upper (resp. lower) half--relaxed limit $\overline u$ (resp. \underline u) is a sub- (resp. super-) solution to the following constrained Hamilton--Jacobi equation:
\begin{equation*} 
	\min\left\{ \partial_t u + \int_\R \left[e^{\frac{\sign(h) f( h )}{f'(0)}\partial_x u} - 1\right] J(h)  \, \dd h + 1   , u \right\}=0, \ \hbox{ on } \ (0,+\infty) \times \R;
\end{equation*}
\item[ii)] the sequence $(\uep)_\e$ converges locally uniformly on $(0,+\infty) \times \R$ to a function $u$ that is Lipschitz continuous with respect to $x$ and continuous in time, and which is a viscosity solution to
\end{itemize}
\begin{equation}\label{eq:ulim}
\begin{cases}
\min\left\{ \partial_t u  + \displaystyle\int_\R \left[ e^{\frac{\sign(h) f( h )}{f'(0)}\partial_x u}-1\right] J(h)  \, \dd h + 1  , u \right\}=0, \ \hbox{ on } (0,+\infty) \times \R, \medskip\\
u(0,\cdot) = u^0.
\end{cases}
\end{equation}
\end{theorem}
We first note that the integral term in~\cref{eq:ulim} is well defined due to the inequality~\cref{ineq:graduep}.  It is also related, up to a change of variables, to the analogous equation obtained by M\'el\'eard and Mirrahimi~\cite[equation~(27)]{SMSM}.  The proof of~\Cref{theo:mut}, appearing in~\Cref{sec:proof_small_mutation}, uses the half--relaxed limits method introduced by Barles and Perthame \cite{barles1990comparison}. It relies heavily upon~\Cref{lem:estim_uep}, which is proved last in \Cref{sec:a_priori}. Under the small mutation regime, we do not use explicit sub- and super--solutions. In particular, the sub- and super--solutions introduced in the propagation regime are not relevant in this situation.

The previous result \Cref{theo:mut} on the behavior of $\uep=-\e\ln(\nep)$ as $\epsilon\to0$ allows us to study the convergence of $n_\e$ as $\epsilon\to0$. Heuristically, when $\epsilon$ is small, we expect that $n_\epsilon \simeq \exp(u(t,x)/\epsilon)$. Thus, the solution $u$ gives an indication on where the solution $n_\epsilon$ is concentrated in the regime of small mutations, or at least at a first order of approximation.
More precisely we obtain the following result:
\begin{theorem}\label{prop:convnullset}
Let $f$ satisfy \Cref{hyp:Jmut} and let $\nep$ be the solution of~\cref{eq:nep_mut} starting with initial condition $n^0_\epsilon$ such that $u^0_\epsilon = -\epsilon\ln(n^0_\epsilon)$ satisfies~\cref{cond:uep0_mutation}.
Then as $\e \to 0$,
\begin{equation}\label{prop:convnep}
\begin{cases}
\nep \to 0 \textit{ locally uniformly in } \mathcal{A}=\{(t,x) \in (0,\infty)  \times \R \ |\ u(t,x)>0\},\medskip\\
\nep \to 1 \textit{ locally uniformly in } \mathcal{B}=\text{Int}\{(t,x) \in (0,\infty)  \times \R \ |\ u(t,x)=0\}.	
\end{cases}
\end{equation}
\end{theorem}
We now discuss, heuristically, \Cref{prop:convnullset}.  When $\epsilon$ is small, we expect that $n_\epsilon \sim 1$ in $\mathcal{B}$, implying that, at time $t$, the phenotype $x$ is realized in the population if $(t,x) \in \mathcal{B}$.  On the other hand, $n_\epsilon \sim 0$ in $\mathcal{A}$.  Similarly, the trait $x$ at time $t$ will not be realized in the population if $(t,x) \in \mathcal{A}$.  Hence, if the mutations are sufficiently small, i.e.~if $\epsilon$ is small enough, the sets $\mathcal{A}$ and $\mathcal{B}$ determine which phenotypes are realized in the population. We can thus deduce the form of the solution $n_\epsilon$ when $\epsilon$ is small. This final result is proved in \Cref{sec:conv_nulset_mut}.

To illustrate \Cref{theo:mut}, we discuss the following example which is also discussed in the paper of Mirrahimi and M\'el\'eard \cite{SMSM} in the case of the fractional Laplacian. We show how our results \Cref{theo:mut} and \Cref{prop:convnullset} gives an approximation of the behaviour of the solution $n_\epsilon$ of the problem~\eqref{eq:nep_mut}.

\begin{example}
Let $\nep$ be the solution of~\eqref{eq:nep_mut} starting with $\nep^0(x)=J(x)^{A/\epsilon},$ then the initial condition $u^0_\epsilon(x) = -\epsilon\ln(\nep^0(x)) = Af(x)/\epsilon$.  Notice that $u^0_\e$ satisfies~\cref{cond:uep0_mutation} because $f$ is concave. It follows from the \Cref{theo:mut} that $u_\epsilon$ converge locally uniformly to the unique viscosity solution of \begin{equation*}
\begin{cases}
\min{\left( \partial_t u  + H(\partial_x u)  , u \right)}=0, \ \hbox{ on } (0,+\infty) \times \R, \medskip\\
u(0,\cdot) = Af.
\end{cases}
\end{equation*}
where the Hamiltonian $H$ is defined by 
\begin{equation*}
H(p) =   \int_\R \left[e^{\frac{\sign(h) f(h)}{f'(0)}p} -1 \right] J(h)  \, \dd h  + 1.
\end{equation*}
From a Taylor expansion, one can check that there exists two positive constants $\overline{\kappa}$ and $\underline{\kappa}$ such that
\begin{equation*}
\begin{split}
	 1+\underline{\kappa} p^2 &= 1+ p^2 \int_0^\infty{ \left(\frac{f(h)}{f'(0)}\right)^2 e^{-\frac{f(h)}{f'(0)}A} J(h)  \, \dd h}\\
 &\leq H(p)
 \leq 1+p^2 \int_0^\infty{ \left(\frac{f(h)}{f'(0)}\right)^2e^{\frac{f(h)}{f'(0)}A} J(h)  \, \dd h} = 1+\overline{\kappa}p^2.
 \end{split}
\end{equation*}
Using these inequalities, we obtain the following estimates of $u$:
\begin{equation*}
  \max \!{\left(\! \inf_{y\in\R}\!{\!  \left( \! Af(y)\! +\! \frac{|x-y|^2}{4\underline{ \kappa}t} -t \right) }, 0\right)} 
   \leq u(t,x) 
   \leq \max \! {\left( \inf_{y\in\R}{ \left( \!Af(y) \! + \! \frac{|x-y|^2}{4\overline{\kappa}t} \! - \! t \right) }, 0\right)}
\end{equation*}
We then deduce that
\begin{equation*}
\begin{split}
   \left\{ (t,x)\in(0,+\infty)\times\R: \ |x|\leq \max_{r\in[0,1]}{\Big(2\sqrt{\underline\kappa}rt + f^{-1}(t(1-r^2)/A)  \Big)} \right\} 
   \subset \{u=0\} \\
   \subset \left\{ (t,x)\in(0,+\infty)\times\R: \ |x|\leq \max_{r\in[0,1]}{\Big(2\sqrt{\overline \kappa}rt + f^{-1}(t(1-r^2)/A)  \Big)} \right\}
\end{split}
\end{equation*}
Combining these estimates with~\Cref{prop:convnullset}, we conclude that the population propagates in the phenotype space to be of  order $f^{-1}(t/A)$ for large time. Roughly, when the mutations in~\cref{eq:nep_mut} are sufficiently small and time is sufficiently large, we see that the phenotypes $x$ realized in the population are those for which $x \lesssim f^{-1}(t/A)$ and those that have not been realized are those for which $x \gtrsim f^{-1}(t/A)$. 
Moreover, we can deduce the following approximation at the first order of $\epsilon$. Formally, using the convergence \Cref{theo:mut}, we can say that the solutions $n_\epsilon=\exp(-u_\epsilon/\epsilon)$ of~\cref{eq:nep_mut} can be approximated by $\exp(-u/\epsilon)$. 
Roughly, we conclude that $n_\e(t,x) \simeq 1$ when $|x| \lesssim f^{-1}(t/A)$ and that $n_\e(t,x) \simeq 0$ when $|x| \gtrsim f^{-1}(t/A)$, up to some additional small error depending on $\e$.
\end{example}

\begin{remark}
Throughout the paper, we use $C$ to refer to any constant depending only on the kernel $J$.  This constant may change line-by-line.
\end{remark}
%

\section{The propagation result: proof of Theorem~\ref{thm:u_conv}}$ $\\

To prove Theorem~\ref{thm:u_conv}, we construct sharp explicit sub- and super--solutions to the non-rescaled problem~\eqref{eq:main}. Our sub- and super--solutions are sharp enough to converge after rescaling to the same solution. This construction is defined in the next proposition.

\begin{proposition}\label[proposition]{prop:subsuper}
Let the kernel $J$ satisfy ~\cref{hyp:J} and $n$ be a solution of~\cref{eq:main} with initial data $n^0$ satisfying~\cref{cond:n0_likeJ}. Then, there exists a bounded positive function $\theta,$ which only depends on $J,$ such that $\theta(s) \to 0$ as $s\to+\infty$, and positive constants $\underline C < 1  < \overline C$, such that, for all $t>0$ and $x\in\R,$
\begin{equation}\label{eq:subsuper}
\ds \dfrac{\underline{C}\exp\left( -\int_0^{t}\theta(s) ds \right)}{1 + e^{-t}/J(x)} \leq n(t,x) \leq \dfrac{2\overline{C}\exp\left( \int_0^{t}\theta(s) ds \right)}{1 + e^{-t}/J(x)}.
\end{equation}
\end{proposition}
The function $\theta$ could, for a specific kernel, be computed explicitly; however, such a computation would be quite involved. For our purposes, we only need to know that it converges to $0$ as $t\to+\infty.$ 
Next let us give some interpretation of this result and an insight into the underlying ideas of the proofs. Let us first rewrite the estimate as 
\[
  \underline{C}\exp\left( -\int_0^{t}\theta(s)\,ds\right) \phi(t,x) \leq n(t,x) \leq 2\overline{C}\exp\left( \int_0^{t}\theta(s) ds \right)\phi(t,x).
\]
where we define the function $\phi(t,x):= (1+e^{-t}/J(x))^{-1}.$ We observe that the behaviour of $n$ is well approximated by the solution of the family of decoupled ODEs:
\begin{equation*}
	\begin{cases}
              \displaystyle\frac{d\phi(t,x)}{dt}  = \phi(1-\phi), \\
              \phi(0,\cdot)  = \frac{J}{1+J} \leq J,
	\end{cases}
\end{equation*}
parametrized by $x\in\R.$ In other words, the behaviour of $n$ at large time is dominated by the reaction term; that is to say that the dispersion term plays a negligible role, in some sense, compared to the growth by reaction.

Before embarking on the proof of this proposition, we explain how \Cref{prop:subsuper} implies \Cref{thm:u_conv}.
\subsection{Proof of \cref{thm:u_conv} assuming \cref{prop:subsuper}}
\begin{proof}[{\bf Proof of \Cref{thm:u_conv}}]
Let us assume that the estimate~\cref{eq:subsuper} holds true for a solution $n$ of~\cref{eq:main} with initial data $n^0$ satisfying~\cref{cond:n0_likeJ}. Then, for any $\epsilon>0,$ the rescaled solution $\nep(t,x) = n(t/\epsilon,\psi_\epsilon(x))$ with $\psi_\epsilon$ defined in~\cref{eq:rescaling_psi}, satisfies for all $t>0$ and $x\in\R$:
\begin{equation*}
\frac{\underline{C}\exp\left( -\int_0^{\frac{t}{\eps}}\theta(s) ds \right)}{1 + e^{-t/\e}/J(x)^\frac{1}{\eps}} \leq n_\eps(t,x) \leq \frac{2\overline{C}\exp\left( \int_0^{\frac{t}{\eps}}\theta(s) ds \right)}{1 + e^{-t/\e}/J(x)^\frac{1}{\eps}},
\end{equation*}
Thus, its Hopf--Cole transformation $u_\epsilon(t,x) = -\epsilon\ln(\nep(t,x))$ satisfies for all $t>0$ and $x\in\R$
\begin{equation}\label{chris1}
\begin{split}
- \eps \ln(\underline{C}) + &\eps \int_0^{\frac{t}{\eps}} \theta (s) ds + \eps \ln\left( 1 + \left( e^{-t}/J(x)\right)^\frac{1}{\eps} \right)   \geq  \ u_\eps(t,x) \\
	&  \geq - \eps \ln(2\overline{C}) - \eps \int_0^{\frac{t}{\eps}} \theta (s) ds + \eps \ln\left( 1 + \left( e^{-t}/J(x)\right)^\frac{1}{\eps} \right).
\end{split}
\end{equation}
We point out that $t^{-1} \int_0^t \theta(s) ds \to 0$ as $t\to\infty$ since $\theta(t) \to 0$ as $t\to +\infty$.
It follows that, locally uniformly in $t$,
\begin{equation}\label{chris2}
\lim_{\eps \to 0} \eps \int_0^{\frac{t}{\eps}} \theta (s) ds = 0.
\end{equation}
In addition, it is easy to see that, locally uniformly in $x$ and $t$, 
\begin{equation}\label{chris3}
\lim_{\e \to 0} \eps \ln\left( 1 + \left( e^{-t}/J(x)\right)^\frac{1}{\eps} \right) = \max\left( f(\vert x \vert) - t, 0 \right).
\end{equation}
Hence, using~\cref{chris1,chris2,chris3}, we see that, locally uniformly in $x$ and $t$,
\[
	\lim_{\e\to0} u_\e(t,x) = \max\{f(x) - t, 0\},
\]
which concludes the proof of ~\cref{thm:u_conv}.
\end{proof}
\begin{remark}
Using the work below, one could, in practice, compute $\theta$ and determine for which kernels $J$ the function $\theta$ is integrable.  When this is the case, the estimate given by \cref{eq:subsuper} is more precise since $\int_0^t \theta(s) ds$ could be replaced by a constant on both sides of the equation. One could then quantify and compare more precisely the expansion of the $\lambda$-level lines of the solution $n$ for various values of $\lambda$.  Further, by plotting the function $x \mapsto \left( 1 + e^{-t}/J(\vert x \vert) \right) n(t,x)$ for various values of time (results not shown), one can investigate the accuracy of the upper and lower bounds given by \cref{eq:subsuper}. The threshold for integrability of $\theta$ appears to be kernels like $\exp\{\sqrt{\vert x \vert}\}$: those which are fatter yield an integrable $\theta$.  In this case,~\cref{eq:subsuper} gives a sharp estimate, up to the constants, of $n$ and, in turn, on the expansion of the level sets of $n$.  On the other hand, when kernels are thinner, $\theta$ is not integrable and~\cref{eq:subsuper} is no longer an accurate point-wise bound, though it is good enough for our purposes.  This is consistent with the fact that when the kernel is thin-tailed the qualitative behavior of $n$ is quite different.
\end{remark}

\subsection{\!\!The existence of sub- and super-solutions: \!Proof of \Cref{prop:subsuper}}
\begin{proof}
We will show that the left hand side and the right hand side of~\cref{eq:subsuper} are respectively a sub- and a super--solution of~\cref{eq:main}. As already observed these sub- and super--solution are constructed from the family of solutions of ODEs of the form
\begin{equation}\label{eq:defphi}
\phi(t,x) = \frac{1}{1 + e^{-t}/J(x)}. 
\end{equation}
Then, we may write our (potential) sub- and super-solutions as
\begin{equation*}
\underline{\phi}(t,x) = \underline{C} \phi(t,x) \exp\left(- \int_0^t  \theta (s) ds \right) \quad \hbox{ and } \quad \overline\phi(t,x) = 2 \ \overline{C}\phi(t,x) \exp\left( \int_0^t  \theta (s) ds \right),
\end{equation*}
for all $t>0$ and $x \in \R$.  Note that, 
for all $x\in\R$ 
\[
   \underline{\phi}(0,x) = \frac{\underline{C}J(x)}{1+J(x)} \leq \underline{C}J(x) \leq n^0(x) \leq \overline{C}J(x) \leq \frac{2 \ \overline{C}J(x)}{1+J(x)} = \overline\phi(0,x).
\]
Moreover, a direct computation shows that the function $\overline\phi$ satisfies, for all $t>0$ and $x\in\R$,
\begin{equation*}
\partial_t\overline\phi(t,x) = \overline\phi(t,x) \left( 1 - \phi(t,x) \right) + \theta(t) \overline\phi(t,x) \geq \overline\phi(t,x) \left( 1 - \overline\phi(t,x) \right) + \theta(t) \overline\phi(t,x). 
\end{equation*}
Then,
\begin{multline*}
   \partial_t\overline\phi(t,x)  -  \overline\phi(t,x) \left( 1 - \overline\phi(t,x) \right) - \left( J*\overline\phi - \overline\phi\right)(t,x)  \\ \geq  \Big( \theta(t) \phi(t,x) - \left( J*\phi - \phi\right)(t,x) \Big) 2\ \overline{C} \exp\left( \int_0^t  \theta (s) ds \right).
\end{multline*}
Thus, if  $  \theta \phi - (J* \phi - \phi) \geq 0$ in $(0,\infty)\times\R,$ the function $\overline\phi$ is a super--solution to \cref{eq:main}. Similarly, if $- \theta \phi - (J*\phi - \phi) \leq 0$ in $(0,\infty)\times\R,$ the function $\underline{\phi}$ is a sub--solution to \cref{eq:main}. The proof of \cref{prop:subsuper} then reduces to proving that
\begin{equation}\label{chris11}
-\theta \phi \leq J* \phi - \phi \leq \theta \phi.
\end{equation}

In the following sections, we obtain upper and lower bounds on this convolution term, completing the proof of \Cref{prop:subsuper}. To do so, we will split the space into two regions depending on time; the large range region $\mathcal{E}_\ell(t)$ and the short range region $\mathcal{E}_s(t)$ defined, for all $t>0$, by
\begin{equation}\label{eq:Et}
\mathcal{E}_\ell(t)=\left\lbrace x\in\R :  |x|\geq f^{-1}(t)\right\rbrace \qquad \hbox{ and } \qquad \mathcal{E}_s(t)=\R\setminus\mathcal{E}_\ell(t).
\end{equation}
 We immediately notice that both regions are preserved by the scaling $\left(t/\epsilon, \psi_\epsilon(x)\right)$. 
 We shall estimate $J*\phi - \phi$ in both regions $\mathcal{E}_\ell(t)$ and $\mathcal{E}_s(t)$ separately.

\subsection{Establishing \Cref{prop:subsuper}: the proof of the bound~\cref{chris11}}$ $

To estimate the convolution term, regardless the region in which $x$ lies, we split the domain of integration of the convolution term as follows
\begin{align*}
& \left( J*\phi - \phi \right) (t,x) \\ & \ds = \int_{|x-y|\leq\gamma(t)}\!\!\!J(x-y)\left( \phi(t,y) - \phi(t,x) \right)\,dy + \int_{|x-y|\geq\gamma(t)} \!\! J(x-y) \left( \phi(t,y) - \phi(t,x) \right)\,dy ,\\
                                 & \ds = I_1(t,x) + I_2(t,x) ,      
\end{align*}
for a function $\gamma(t)$ to be determined that localizes the integral around $x$. In what follows, we choose $\gamma(t)$ to be a positive and increasing function of time $t$ such that $\gamma(t)\leq f^{-1}(t)$. Note that by symmetry of the problem, we can assume that $x \geq 0$, which we do from now on.

The existence of $\theta$ in~\cref{chris11} is equivalent to showing that, for $k = 1,2$,
\begin{equation}\label{chris12}
	\lim_{t\to\infty} \frac{|I_k(t,x)|}{\phi(t,x)} = 0,
\end{equation}
where the limit holds uniformly in $x$.

\subsubsection{Estimation of the integral $I_1$}
In the region where $y$ is close to $x$, that is $|x-y|\leq\gamma(t)$, we estimate the difference $\phi(t,y) - \phi(t,x)$ by using the Taylor expansion of $\phi$ around location $x$. 

More precisely, for any $t>0,$ $x\in\R$ and $y$ such that $|x-y|\leq\gamma(t),$ there exists $\xi_{t,x,y}\in(0,1)$ such that 
\begin{equation}\label{eq:phi_Taylor}
   \phi(t,y) - \phi(t,x) = (y-x)\partial_x\phi(t,x\xi_{t,x,y}+(1-\xi_{t,x,y})y)\,.
\end{equation}
For notational ease, we omit the subscripts $(t,x,y)$ for $\xi_{t,x,y}$ in the sequel. Plugging this expression into $I_1$ we obtain
\begin{equation*}
I_1   = \int_{|x-y|\leq\gamma(t)} J(x-y) (y-x)\partial_x\phi(t,x\xi+(1-\xi)y) \,dy.
\end{equation*}
To estimate the first derivative of $\phi$, we use the following.
\begin{lemma}\label{lem:d2phi}
   There exists a positive function $\theta_1(t)$, depending only on $J,$ such that, for any $z\in\R$ and $t>0$, 
   \begin{equation*}
      |\partial_x\phi(t,z)| \leq \theta_1(t)\,\phi(t,z),
   \end{equation*}
   and such that $\theta_1(s) \to 0$ as $s\to\infty$.
\end{lemma}

\begin{proof}
Using the form of $J$
, we can rewrite $\phi(t,z)= (1 + e^{f(z)-t})^{-1}$. A direct computation shows that, for all $t>0$ and $z\in\R$,
\begin{align*}
  |\partial_x \phi(t,z)|  & \ds = \left| \phi(t,z)  f'(z) \frac{e^{f(z)-t}}{1+e^{f(z)-t}}\right|,
\end{align*}
so that 
 \begin{align*}	
    \left\vert \frac{\partial_x\phi}{\phi} \right\vert(t,z)   &\leq  \ds  \frac{e^{f(z)-t}}{1+e^{f(z)-t}}  f'(z)\,. 
 \end{align*}     
To estimate the right hand side of the inequality above, fix $\alpha\in(0,1)$. 
 First,  consider the case when $|z|\leq f^{-1} \left( \alpha t  \right)$.  Then $e^{f(z) -t} \leq e^{-(1-\alpha)t}$, and, hence
\begin{equation*}
    \left\vert \frac{\partial_x\phi}{\phi} \right\vert(t,z)  \leq  \ds  e^{-(1-\alpha)t} \Vert f' \Vert_{\infty}. 
\end{equation*}
On the other hand if $|z|\geq f^{-1}(\alpha t )$, then 
$\frac{e^{f(z)-t}}{1+e^{f(z)-t}}\leq 1$  and, consequently,
\begin{equation*}
    \left\vert \frac{\partial_x\phi}{\phi} \right\vert(t,z)  \leq  \ds  \sup_{|z|\geq f^{-1}(\alpha t)} \vert  f'(z) \vert. 
\end{equation*}
Defining
\begin{equation}\label{eq:theta_1}
	\theta_1(t) = \max\left\{ \|f'\|_\infty e^{-(1-\alpha)t},  \sup_{|z|\geq f^{-1}(\alpha t)}  \vert f'(z) \vert \right\},
\end{equation}
 we conclude that, for all $t>0$ and $z\in\R$ 
\begin{equation*}
    \left\vert \frac{\partial_x\phi}{\phi} \right\vert(t,z)  \leq  \theta_1(t).
\end{equation*}
The convergence of $\theta_1(s)$ to zero as $s$ tends to infinity is clear from the definition and the assumptions on $f$ in \Cref{hyp:J}.
This concludes the proof of Lemma~\ref{lem:d2phi}.
\end{proof}
We deduce from \cref{lem:d2phi}, an estimate on $I_1$
\begin{equation*}
|I_1|   \leq  \theta_1(t) \, \int_{|x-y|\leq\gamma(t)} J(x-y)  \vert y-x \vert\phi(t,x\xi+(1-\xi)y)\,dy.
\end{equation*}
Changing variables, we must estimate
\begin{equation}\label{eq:I1}
\int_{|h|\leq\gamma(t)} J(h) \vert h \vert \phi(t,x+(1-\xi)h)\,dh.
\end{equation}
To do so, we bound $\phi(t,x+(1-\xi)h)$ with $\phi(t,x)$ by choosing the function $\gamma(t)$ carefully.
We first prove that we may choose the positive function $\gamma(t)$ such that, for all $x\in\mathcal{E}_\ell(t)$, $h\in(-\gamma(t),\gamma(t))$, and $\xi\in(0,1)$,
\begin{equation}\label{chris10}
	\phi(t,x+(1-\xi)h)\leq e \phi(t,x).
\end{equation}
Indeed, we have
\begin{equation*}
\begin{array}{ll}
\ln\left( \frac{\phi(t,x+(1-\xi)h)}{\phi(t,x)} \right) 
        & \leq \vert \ln(\phi(t,x+(1-\xi)h)) - \ln(\phi(t,x)) \vert, \\
        & \leq \sup_{z\in[x - (1-\xi)|h|,  x + (1-\xi)|h|]} \left\vert\partial_x (\ln \phi)(t,z)\right\vert \vert h \vert.
\end{array}
\end{equation*}
Using the bound on $\partial_x \left( \ln \phi\right)$ of \Cref{lem:d2phi}, we have
\[
	\ln\left( \frac{\phi(t,x+(1-\xi)h)}{\phi(t,x)} \right)\leq \exp( \theta_1(t) \gamma(t) ).
	\]
Fixing $\gamma$ such that $\theta_1 \gamma \leq 1$ yields~\cref{chris10}. For technical reasons, discussed in the proof of the estimate of $I_2$, we define 
\begin{equation}\displaystyle\label{def:newgamma}
\gamma(t) =
	\min  \left\{f'\left(\frac{f^{-1}(t)}{2}\right)^{-1},\theta_1(t)^{-1} \right\}
\end{equation}
It is clear that $\theta_1 \gamma \leq 1$, as desired.
We point out that, from~\cref{eq:theta_1},
\begin{equation*}
\theta_1(t)^{-1}=\min \left( e^{(1-\alpha)t}, \left[ f'(f^{-1}(\alpha t)) \right]^{-1}  \right).
\end{equation*}
Our choice of $\gamma$ implies
\begin{equation*}
|I_1| \leq e\theta_1(t) \, \left( \int_{|h|\leq\gamma(t)}  J(h) \vert h \vert \,dy \right) \phi(t,x).
\end{equation*}
We now show that 
\begin{equation}\label{eq:theta1m1}
\lim_{t \to + \infty} \theta_1(t) \int_{|h|<\gamma(t)}J(h)|h|dh = 0.
\end{equation}
From our assumption \cref{hyp:J_poly} in \Cref{hyp:J} on $J$, there exists $H>0$ such that if $|h|>H$, then $J(h)\leq |h|^{-\mu}$. We deduce that there exists a constant $C$, depending only on $J$ such that, for all $t>0$,
\begin{equation*}
\begin{split}
	\theta_1(t) \int_{|h|<\gamma(t)}J(h)|h|dh & \ds \leq \theta_1(t)\int_{|h|<H}J(h)|h|dh + \theta_1(t)\int_{H<|h|<\gamma(t)}J(h)|h|dh \\
	&\leq \theta_1(t)C + \theta_1(t)\int_{H<|h|<\gamma(t)}|h|^{1-\mu}dh,
\end{split}
\end{equation*} 
where we interpret the last integral to be zero in the case when $\gamma(t) \leq H$.
Thus we get 
\begin{equation*}                                               
\theta_1(t) \int_{|h|<\gamma(t)}J(h)|h|dh \leq 
\begin{cases}
	C\left(\theta_1(t) + \theta_1(t)|\gamma(t)|^{2-\mu}\right), \qquad& \text{ if } \mu \neq 2,\\[1 mm]
	C\left(\theta_1(t) + \theta_1(t)\ln(\gamma(t))\right),    \qquad& \text{ if } \mu = 2.
\end{cases}
\end{equation*}
Using the definition of $\gamma$,~\cref{def:newgamma}, this inequality becomes
\begin{equation*}                                               
\theta_1(t) \int_{|h|<\gamma(t)}J(h)|h|dh \leq 
\begin{cases}
	C\left(\theta_1(t) + \theta_1(t)^{\mu-1}\right), \qquad& \text{ if } \mu \neq 2,\\[1mm]
  C\theta_1(t)\left(1 + \ln\left(\theta_1(t)^{-1}\right)\right),   \qquad& \text{ if } \mu = 2.
\end{cases}
\end{equation*}
In both cases, we have established that $|I_1|/\phi(t,x) \to 0$ as $t \to \infty$, as claimed.  This establishes~\cref{chris12} for $k = 1$.

\subsubsection{Estimation of the integral $I_2$} The arguments for the upper and lower bounds are different.  As it is simpler, we prove the lower bound first.  Since $\phi$ is positive,
\[
	I_2(t,x) \geq -\int_{|x-y|\geq \gamma(t)} J(x-y) \phi(t,x) dx
		= - \phi(t,x) \int_{|h| \geq \gamma(t)} J(h) dh.
\]
Clearly $\liminf_{t\to\infty} I_2/ \phi(t,x) \geq 0$.

\noindent This finishes the proof of the lower bound for $I_1$ and $I_2$.  To conclude, we need only obtain a matching upper bound  of $I_2$ in order to obtain~\cref{chris12} and thus \Cref{prop:subsuper}.  The proof of this bound is somewhat involved. 
We break up our estimates based on whether $x$ is in the short-range or long-range regime.

\paragraph{\# The long range region:}
We handle first the case when $x \in\mathcal{E}_\ell(t),$ that is $x\geq f^{-1}(t)$. We first split the integral $I_2$ as
\begin{equation}\label{chris13}
\begin{split}
      I_2 &= \int_{\substack{|x-y|\geq\gamma(t)\\ |y|\geq |x|}}\!\!\!\!\! J(x-y)\big( \phi(t,y) - \phi(t,x) \big)\,dy
                + \int_{\substack{|x-y|\geq\gamma(t)\\ |y|\leq |x|}} \!\!\!\! J(x-y)\big( \phi(t,y) - \phi(t,x) \big)\,dy , \\
          & \leq \left( \int_{\substack{|x-y|\geq\gamma(t)\\ |y|\geq |x|}} J(x-y) \frac{\phi(t,y)}{\phi(t,x)} \,dy
                + \int_{\substack{|x-y|\geq\gamma(t)\\ |y|\leq |x|}} J(x-y) \frac{\phi(t,y)}{\phi(t,x)} \,dy \right) \phi(t,x), \\
          & \leq \left( II_1 + II_2 \right) \phi(t,x).      
\end{split}
\end{equation}
The first part of the integral, $II_1$, is estimated using the monotonicity of $\phi$ in the spatial variable to obtain
\begin{equation*}
      II_1 \leq \int_{\substack{|x-y|\geq\gamma(t)\\ |y|\geq |x|}} J(x-y)\,dy \leq \int_{|h|\geq\gamma(t)} J(h)\,dh.
\end{equation*}

We now turn to the second part term in~\cref{chris13}, $II_2$. Recall that we are assuming, by the symmetry of the problem that $x$ is positive.  We  decompose the integral in four pieces, one integral close to $y=-x$, two integrals close to $y=0$ and the last one centered around $|y|=x/2.$ More precisely,
\begin{equation*}
\begin{split}
      II_2 & = \int_{\substack{|x-y|\geq\gamma(t)\\ |y|\leq |x|}} \! \! \! J(x-y)\frac{\phi(t,y)}{\phi(t,x)} dy 
            \leq \int_{-x}^{-\gamma(t)} \!\!\!\!\!\!\!\!\! J(x-y)\frac{\phi(t,y)}{\phi(t,x)}\,dy   
                 + \int_{-\gamma(t)}^{0} \!\!\!\!\!\! J(x-y)\frac{\phi(t,y)}{\phi(t,x)}\,dy \\[0.01cm]
           &   +  \int_{0}^{\gamma(t)}\!\!\!\!\!  J(x-y)\frac{\phi(t,y)}{\phi(t,x)}\,dy +  \int_{\gamma(t)}^{x-\gamma(t)} \!\!\! \!\!\!\!\! J(x-y)\frac{\phi(t,y)}{\phi(t,x)}\,dy  
           =III_1 + III_2+III_3+III_4.
\end{split}
\end{equation*}
We now estimate each of the four integrals $III_1, III_2, III_3, III_4$ in turn and show that each tends to 0 in the limit $t\to\infty$.  

Let us first notice that for all $x,y\in\R$ and $t>0,$ 
\begin{equation}\label{chris14}
   J(x-y)\frac{\phi(t,y)}{\phi(t,x)} = \frac{J(x-y)J(y)}{J(x)} \frac{J(x) +e^{-t}}{J(y)+e^{-t}}.
\end{equation}
In addition, for all $x\in\mathcal{E}_\ell(t)$ and  we  have on the one hand
\[
J(x)\leq e^{-t},
\]
and since $y\in(-x,x)$,
\[
  \frac{J(x) +e^{-t}}{J(y)+e^{-t}}\leq 1. 
\]
To proceed further, we require the following useful fact, in which we see the need for the intricate description in~\cref{def:newgamma}.  From its definition, it is clear that if $t>0$ and $x \in \mathcal{E}_\ell(t)$, then $2\gamma(t) \leq x$.

\paragraph{\#\# Estimation of $III_1$:}
First we estimate $III_1$.  Due to the limits of integration, we have that $|x-y| \geq |x|$, giving
\begin{align*}
   III_1 = \int_{-x}^{-\gamma(t)} {  \frac{J(x-y)J(y)}{J(x)} \frac{J(x) +e^{-t}}{J(y)+e^{-t}} \,dy }  
           \leq \int_{-x}^{-\gamma(t)}  J(y) \,dy  
           \leq  \int_{\gamma(t)}^{\infty} J(y) \,dy.
\end{align*}
Since $J$ is integrable and since $\gamma(t) \to \infty$ as $t\to\infty$, this tends to zero as $t$ tends to infinity.

\paragraph{\#\# Estimation of $III_2$.}
Next we estimate $III_2$.  Due to the limits of integration $J(x-y) \leq J(x)$.  Using this and~\cref{chris14}, we have
\begin{equation*}
\begin{split}
   III_2 &= \int_{-\gamma(t)}^{0} \frac{J(x-y)J(y)}{J(x)} \frac{J(x) +e^{-t}}{J(y)+e^{-t}}\,dy 
         \leq \int_{-\gamma(t)}^{0}  \frac{J(x-y)}{J(x)} (J(x)+e^{-t}) \,dy \\
          &\leq \int_{-\gamma(t)}^{0}  (J(x)+e^{-t}) \,dy
          \leq \left(J(x) + e^{-t}\right) \gamma(t)
          \leq 2 e^{-t} \gamma(t) \,.
 \end{split}
\end{equation*}
In the last step we used that $x \geq f^{-1}(t)$ so that $J(x) \leq e^{-t}$.
Then $III_2$ tends to zero as $t \to \infty$ because, by construction, 
 $\gamma(t) \leq e^{(1-\alpha)t}$.

\paragraph{\#\# Estimation of $III_3$.}
To estimate the third integral, $III_3$, we first notice that $J(x- y) \leq J(x-\gamma(t))$ for all $y$ in the domain of integration $[0,\gamma(t)]$.  Here we are using the definition of $\gamma$~\cref{def:newgamma} and the fact that $x \geq 2 \gamma(t)$, observed above. Using this inequality and~\cref{chris14}, again, we obtain
\begin{equation}\label{chris15}
   III_3 = \int_{0}^{\gamma(t)}  \frac{J(x-y)J(y)}{J(x)} \frac{J(x) +e^{-t}}{J(y)+e^{-t}}\,dy 
         \leq \int_{0}^{\gamma(t)} { \frac{J(x-\gamma(t))}{J(x)} (J(x)+e^{-t}) \,dy }.
\end{equation}
The main difficulty is in obtaining a bound, independent of time, of $J(x-\gamma(t))/J(x)$.  We obtain this bound now.
From the definition of $\gamma(t)$~\cref{def:newgamma}, we see that 
\begin{equation}\label{chris16}
\begin{array}{ll}
\dfrac{J(x-\gamma(t))}{J(x)} & = \exp\left\{f(x)-f\left(x-\gamma(t)\right)\right\}, \\
                     & \ds \leq \exp\left\{\max_{z \in[x-\gamma(t),x]} f'(z)\gamma(t)\right\},\\
                    & \ds \leq \exp\left\{\max_{z \in[f^{-1}(t)/2,x]} f'(z)\gamma(t)\right\}
                     \leq \exp(f'(f^{-1}(t)/2)\gamma(t))\leq e.
\end{array}	
\end{equation}
The last step follows from the fact that $x \geq f^{-1}(t)$ and that $x -\gamma(t) \geq x/2$, which follows from the inequality $x \geq 2\gamma(t)$.

When $t$ is small, $\gamma(t)$ is bounded; hence, $\max_{z\in[\gamma(t),x]}f'(z) \gamma(t)$ is bounded.  From this and~\cref{chris16}, it follows that $J(x-\gamma(t))/J(x) \leq C$, for some constant $C$ independent of time.  When $t$ is sufficiently large, the concavity of $f$, given in \Cref{hyp:J}, implies that $f'$ is monotonic so that~\cref{chris16} becomes
\begin{equation}\label{chris17}
	\frac{J(x-\gamma(t))}{J(x)} \leq \exp\left\{ f'\left(\frac{f^{-1}(t)}{2}\right) \gamma(t)\right\}.
\end{equation}
From the definition of $\gamma$,~\cref{def:newgamma}, we have that $\gamma(t) \leq f'(f^{-1}(t)/2)^{-1}$.  Putting this together with~\cref{chris17} yields
\[
	\frac{J(x-\gamma(t))}{J(x)} \leq e.
\]

Hence, there is a constant $C$, indepedent of time, such that $J(x-\gamma(t))/J(x) \leq C$. Using this in our estimate~\cref{chris15}, we get the bound
\[
	III_3 \leq C(J(x) + e^{-t})\gamma(t)
		\leq C(J(2\gamma(t)) + e^{-t}\gamma(t)).
\]
Using \Cref{hyp:J} and our arguments above, we see that the right hand side tends to zero, as desired.

\paragraph{\#\#  Estimation of $III_4$}
For all $y$ in the domain of integration $[\gamma(t),x-\gamma(t)]$, we have that $y \leq x$ so that $J(x) \leq J(y)$.  Re-arranging the fourth integral with this inequality yields
\begin{equation}\label{chris20}
\begin{split}
   III_4 &= \int_{\gamma(t)}^{x-\gamma(t)} \frac{J(x-y)J(y)}{J(x)} \frac{J(x) +e^{-t}}{J(y)+e^{-t}}\,dy
   	\leq \int_{\gamma(t)}^{x-\gamma(t)}  \frac{J(x-y)J(y)}{J(x)} \,dy\\
	&= \int_{\gamma(t)}^{x/2}  \frac{J(x-y)J(y)}{J(x)} dy + \int_{x/2}^{x-\gamma(t)}  \frac{J(x-y)J(y)}{J(x)} dy
           = 2 \int_{\gamma(t)}^{x/2}  \frac{J(x-y)J(y)}{J(x)} \,dy,
\end{split}
\end{equation}
where we have changed the variable $y \mapsto x-y$ to obtain the last equality. Notice that $\gamma(t)\leq x/2 \leq x- \gamma(t)$ due to our observation that $x \geq 2\gamma(t)$. To estimate the last term in~\cref{chris20}, we need to distinguish between the cases when $xf'(x)$ is bounded and when $xf'(x)$ is unbounded.

If $x f'(x)$ is bounded, we have
\[
	III_4
		\leq 2\frac{J(x/2)}{J(x)} \int_{\gamma(t)}^{x/2} J(y)dy
		\leq 2\frac{J(x/2)}{J(x)} \int_{\gamma(t)}^\infty J(y)dy,
\]
because $y \in (\gamma(t), x/2)$ and thus $J(x-y) \leq J(x/2)$.  By Taylor's theorem, we have that
\[
	\frac{J(x/2)}{J(x)}
		= \exp\left\{ f(x) - f(x/2)\right\}
		\leq \exp\left\{ \sup_{\xi \in (x/2,x)} f'(\xi) \frac{x}{2}\right\}.
\]
Using the eventual concavity of $f$ along with the boundedness of $f'(\dfrac{x}{2}) \dfrac{x}{2}$, we have that $J(x/2)/J(x) \leq C$.  Hence
\[
	III_4 \leq C \int_{\gamma(t)}^\infty J(y) dy,
\]
which tends to $0$ as $t$ tends to infinity.

If $xf'(x)$ tends to infinity, we require a different argument.  First, notice that, if $J(x-y)J(y)/J(x)$ is bounded above by an integrable function $\overline J$ uniformly in $x$, then we are finished because
\[
	III_4 \leq 2\int_{\gamma(t)}^\infty \overline J(y) dy,
\]
which tends to $0$ as $t$ tends to infinity.  We prove this now.

When $t$ is small, the domain of integration in $III_4$ is bounded as is $J(x-y) J(y)/J(x)$.  Hence, we may restrict to considering only the case when $t$ is large enough that $f'$ is decreasing on $(\gamma(t),\infty)$.

First, notice that
\[
	\frac{J(x-y)J(y)}{J(x)} = \exp \left\{f(x) - f(y) - f(x-y)\right\}
		= \exp \left\{ y\int_0^1 f'(x- ys) ds - f(y)\right\}.
\]
Since $f'$ is decreasing, we have that
\[
	\frac{J(x-y)J(y)}{J(x)} 
		\leq \exp \left\{ yf'(y) - f(y)\right\}.
\]
From \Cref{hyp:J}, there exists $\e_0\in(0,1)$ such that $yf'(y)/f(y) \leq 1 - \e_0$ for $y \geq \gamma(t)$ and $t$ sufficiently large.  Hence, we obtain that
\[
	\frac{J(x-y)J(y)}{J(x)} 
		\leq \exp \left\{ - \e_0 f(y)\right\}.
\]
On the other hand, since $y f'(y)$ tends to infinity, it is clear that $e^{-\e_0f}$ is integrable, finishing the proof that $III_4$ tends to $0$ as $t$ tends to infinity.

In conclusion, we have show that when $x \in \mathcal{E}_\ell(t)$,
\[
	I_2(t,x) = \left(II_1 + (III_1 + III_2 + III_3 + III_4)\right)\phi(t,x)
\]
where $II_1$ and $III_k$ tend to zero as $t \to \infty$ uniformly in $x$ for $k \in \{1,2,3,4\}$.  This concludes the proof in the long range region $x \in \mathcal{E}_\ell(t)$.

\paragraph{\# The short range region $x \in \mathcal{E}_s(t)$.}
We now turn to the simpler case where  $x$ is in the short range region $\mathcal{E}_s(t),$ that is $|x|\leq f^{-1}(t).$  In this region, notice that the function $\phi$ is bounded from below $\phi(t,x)\geq 1/2$. We can estimate directly as follows
\begin{equation*}
  I_2 =  \int_{|x-y|\geq\gamma(t)} J(x-y)\big( \phi(t,y) - \phi(t,x) \big)\,dy \leq 2 \left( \int_{|h|\geq\gamma(t)}  J(h)\,dh \right) \phi(t,x).
\end{equation*}
Of course, $\int_{|h| \geq \gamma(t)} J(h)dh$ tends to $0$ as $t$ tends to infinity.

This concludes the proof of~\cref{chris12} and, thus, the proof of \Cref{prop:subsuper}.
\end{proof}

\section{The propagation regime: the proof of Theorem~\ref{thm:n_conv}}$ $\\

\noindent We now deduce from Theorem~\ref{thm:u_conv} the asymptotic behaviour of $\nep$ as $\epsilon$ tends to $0$.

\begin{proof}[{\bf Proof of \Cref{thm:n_conv}}]
We first look at the limit in any compact subset of $\{u>0\}$, and then we focus on the limit in any compact subset of $\Int\{u = 0\}$.

\subsubsection*{Part (a): convergence of $\nep$ in $\mathrm{Int}\{u>0\}$}
Fix any compact subset $K\subset\{u > 0\}$, there exists a positive constant $\alpha$ such that for all $(t,x) \in K$, we have $u(t,x) > \alpha $.  Due to \cref{thm:u_conv}, we know that $\uep$ converges locally uniformly on $(0,+\infty)\times\R$ to $u(t,x) = \max(f(x)-t,0).$  Hence, for all $\e$ sufficiently small, $\uep(t,x) \geq \alpha/2$ for $(t,x) \in K$.  Then for $\e$ sufficiently small, we have that, for $(t,x) \in K$,
\[
	\nep(t,x) = e^{-\frac{\uep(t,x)}{\e}}
		\leq e^{-\frac{\alpha}{2\e}}.
\]
Taking the limit $\e\to0$ yields the uniform convergence of $\nep$ to zero on $K$.
This proves point (a) of \Cref{thm:n_conv}.

\subsubsection*{Part (b): convergence of $\nep$ in $\mathrm{Int}\{u=0\}$}
First, we use the maximum principle to show that, locally uniformly on $\{t>0\}$,
\begin{equation}\label{eq:below_one}
	\limsup_{\e\to0} \|\nep(t,\cdot)\|_\infty \leq 1.
\end{equation}
Indeed, it follows from~\cref{cond:n0_likeJ} that there exists $\overline C$, independent of $\e$ such that $\nep(0,x) \leq \overline C$ for all $\e>0$.  From~\cref{eq:nep_prop}, $\overline n(t) := 1+ \overline C e^{-t/\e}$ is a super-solution to $\nep$.  Hence, we have that $\nep(t,x) \leq 1+ \overline C e^{-t/\e}$, which establishes the estimate~\cref{eq:below_one}.

Let $K$ be a compact subset of  $\Int\{u = 0\}$.  
Our goal is to show that $\nep=\exp(-\uep/\epsilon)$ converges to $1$ uniformly in $K.$
Recall that $\uep$ is a solution to~\cref{eq:uep_prop}; that is, for all $(t,x) \in \R_+\times \R$,
\begin{equation}\label{eq:neps_LowerBound}
\begin{split}
      \nep(t,x) &= 1 + \partial_t\uep(t,x) - \int_{-\infty}^\infty J(h) \left[ 1 - e^{ - \frac{1}{\e} \left(  \uep \left( t,\psi_\e^{-1}\left( \psi_\e( x ) - h  \right) \right) - u_\e(t,x) \right)}\right] dh.
\end{split}
\end{equation}
We need a lower--bound for the right hand side to conclude. To do so, we follow the approach of \cite{EvansSouganidis,SMSM} which consists of replacing $\uep$ by a well chosen test function to obtain a sharp estimate. Fix any $(t_0,x_0)\in K$.  Then, 
since $K \subset \{u < 0\}$, we may fix $\delta_0>0$ such that $f(x_0) < t_0-\delta_0$. For all $(t,x) \in \R_+\times \R$, let
\begin{equation*}
 \chi(t,x) = \max\left\{0,f(x) - (t_0 -\delta_0)\right\}+(t-t_0)^2.
\end{equation*}
 Since $\chi$ is nonnegative and $u = 0$ on $K$, $u - \chi$ has a maximum at $(t_0,x_0)$, strict and local in $t$, global in $x$. In order to follow the convergence, we also define the following perturbed test function $\chi_\e$ as follows.  For any $(t,x) \in \R_+ \times \R$, let
\begin{equation*}
\chi_\e(t,x):= \e \ln \left( 1 + e^{-\frac{t_0 - \delta_0}{\e}} J( x ) ^{-\frac{1}{\e}} \right) + (t-t_0)^2.
\end{equation*}
Observe that we may reformulate $\chi_\epsilon$ in term of $\phi$, defined in~\cref{eq:defphi}, and the space scaling $\psi_\epsilon$ as
\begin{equation}\label{eq:chi_phi}
	\chi_\epsilon(t,x) = -\epsilon \ln\left( \phi\left( \frac{t_0-\delta_0}{\epsilon} , \psi_\epsilon(x) \right)\right)+ (t-t_0)^2.
\end{equation}

Observe that $\chi_\e (t,x) \to \chi(t,x)$ as $\epsilon\to0$ locally uniformly on $(0,\infty)\times\R.$  Then there exists a sequence $(t_\e,x_\e)$ such that $u_\e - \chi_\e$ has a maximum on $(t_0-r,t_0+r)\times \R$ for some small $r>0$ at $(t_\e,x_\e)$ that is strict in $t$ and such that $t_\e \to t_0$ as $\epsilon\to0$.  We note that the fact that $u_\e - \chi_\e$ has a maximum that is global in $x$ is not immediate from the locally uniform convergence of $\uep$ to $u$; however, it follows easily from \Cref{prop:subsuper}. We now plug our test functions $\chi_\epsilon$ into~\cref{eq:neps_LowerBound} and obtain
\begin{equation}
\begin{split}
& \nep(t,x) = 1+ \partial_t \uep(t,x)\\
	& - \!\!\int_\R \!\! J(h)\left[1 -
                  e^{\! -\! \frac1\e \left(  
                                \left(u_\e - \chi_\e \right)\left(t, \psi_\e^{-1}\left( |\psi_\e( x ) - h| \right) \right)
                                - \left(u_\e - \chi_\e\right) \left(t, x \right)   
                                + \chi_\e\left(t, \psi_\e^{-1}\left(|\psi_\e(x) - h |\right) \right) 
                                - \chi_\e\left(t,x\right)  \right) }\right]\!\!dh.
\end{split}\end{equation}
From the maximum property of $\uep-\chi_\e$ at $(t_\e,x_\e)$, we know that  $\partial_t \chi_\e (t_\e,x_\e)=\partial_t \uep (t_\e,x_\e)=2(t_\e-t_0)$ and that for all $h \in\R$,
\[
   \left(u_\e - \chi_\e \right)\left(t_\e, \psi_\e^{-1}\left( |\psi_\e( x_\e) - h| \right) \right)
                                - \left(u_\e - \chi_\e\right) \left(t_\e, x_\e \right) \leq 0.
\]
Thus, we obtain, at $(t_\e,x_\e)$,
\begin{equation*}
\nep(t_\epsilon,x_\e) \geq 1+ 2(t_\e-t_0) -
             \int_\R J(h)\left[1 - 
                  e^{ - \frac1\e \left(  
                                 \chi_\e\left(t_\e, \psi_\e^{-1}\left( \psi_\e( x ) - h  \right) \right) 
                                - \chi_\e\left(t_\e,x_\e\right)  \right) }\right]
          dh.
\end{equation*}   
Then, the link between $\chi_\e$ and $\phi$ yields
\begin{equation}
\begin{split}
\nep(t_\epsilon,x_\e) & \geq 1+ 2(t_\e-t_0) -
             \int_\R J(h) \left[ 1 - \frac{ \phi\left( \frac{t_0-\delta_0}{\epsilon} ,   \psi_\e( x_\e) - h \right)  }{ \phi\left( \frac{t_0-\delta_0}{\epsilon} ,  \psi_\e( x_\e)  \right)  }
                           \right] dh \\
                            & = 1+ 2(t_\e-t_0)  
                                    + \frac{J*\phi - \phi}\phi \left( \frac{t_0-\delta_0}{\epsilon} ,  \psi_\e( x_\e)\right).
\end{split}
\end{equation}
Using~\cref{chris11}, this implies
\[
	\nep(t_\epsilon,x_\e)
		\geq 1+ 2(t_\e-t_0) - \theta\left(\frac{t_0-\delta_0}{\e}\right),
\]
where we recall that $\theta(s) \to 0$ as $s \to \infty$.  Since $t_\e \to t_0$ as $\e\to\infty$ and since $t_0 - \delta_0 > f(x_0) \geq 0$, we obtain
\begin{equation}\label{chris31}
	\liminf_{\e\to0} \nep(t_\e,x_\e) \geq 1.
\end{equation}
To conclude, we must bootstrap~\cref{chris31} to deduce information about $\nep(t_0,x_0)$. 
By construction of $(t_\e,x_\e)$,
\begin{equation*}
  u_\e(t_\e,x_\e) - \chi_\e(t_\e,x_\e) \geq u_\e(t_0,x_0) - \chi_\e(t_0,x_0),
\end{equation*}
which implies that 
\begin{equation*}
\begin{array}{ccl}
u_\e(t_0,x_0)-u_\e(t_\e,x_\e)  & \leq & \chi_\e(t_0,x_0) -\chi_\e(t_\e,x_\e)  \\
                               & \leq & \chi_\e(t_0,x_0)
                               = \e \ln \left( 1 + e^{-\frac{t_0 - \delta_0}{\e}}\left[ J(\psi_\e(x_0)) \right]^{-1} \right).
\end{array}
\end{equation*}
Since $\nep = \exp(-u_\epsilon/\epsilon)$, we obtain
\begin{equation*}
\nep(x_\e,t_\e) \leq n_\e(x_0,t_0)\left( 1 + \Big(e^{-(t_0 - \delta_0)} J(x_0)^{-1}\Big)^{\frac1\epsilon}\right).
\end{equation*}
Using that $f(x_0) < t_0 + \delta_0$ along with~\cref{chris31}, we obtain
\begin{equation}\label{chris32}
	1 \leq \liminf_{\e\to0} \nep(t_\e,x_\e)
		\leq \liminf_{\e \to 0} \nep(t_0,x_0).
\end{equation}
The combination of~\cref{eq:below_one} and \cref{chris32} concludes the proof.
\end{proof}

\section{The small mutation regime}\label{sec:proof_small_mutation}$ $\\

This section is devoted to the proof of~\Cref{theo:mut}.  We obtain some {\em a priori} estimates on $u_\epsilon$ and $\partial_x u_\epsilon$ in order to take the half--relaxed limits of $u_\epsilon$ to obtain $u$, the solution of~\cref{eq:ulim}. We then use this limit $u$ to estimate the level sets of $\nep$ as $\e \to 0$. With the strategy in mind, we proceed with the proof of \Cref{theo:mut}.

\subsection{Proof of Theorem \ref{theo:mut}}\label{sec:conv_uep_mut}

We start with the proof of \Cref{theo:mut} (i), which we rephrase into the following lemma for legibility.

\begin{lemma}\label[lem]{lem:half_relaxed_limits}
	Let $\underline u$ and $\overline u$ be defined by~\eqref{eq:half_relaxed}.  Then $\underline u$ and $\overline u$ satisfy
	\begin{equation*}
	\begin{split}
		 & \min\left\{ \partial_t \ulbar  +  \int_\R \left[  e^{\frac{\sign(h) f( h )}{f'(0)}\partial_x \ulbar} -1 \right]J(h)  \, \dd h + 1   , \ulbar \right\}\geq 0, \ \text{ and} \\
		 & \min\left\{ \partial_t \overline u + \int_\R\left[ e^{\frac{\sign(h) f( h )}{f'(0)}\partial_x \overline u} -1 \right] J(h)  \, \dd h + 1, \overline u \right\}\leq 0
\end{split}
	\end{equation*}
on $(0,+\infty)\times\R$.	
\end{lemma}

\begin{proof} We first prove that the lower half--relaxed limit $\underline u$ of $\uep$ is a viscosity super--solution to~\cref{eq:ulim}, where we recall that $\underline{u}$ is defined by
\[
  \underline{u} (t,x):= \liminf_{\underset{(s,y)\to (t,x)}{\e\to 0}} \ u_\e(s,y).
\]
First, $\underline u \geq 0$ because $u_\eps \geq 0$ for all $\e$.
Let  $\varphi$ be a test function in  $\mathcal{C}^1 ((0,+\infty)\times\R)$ such that $\underline u-\varphi$ has a strict global minimum equal to $0$ at some point $(t_0,x_0)$ with $t_0>0$.  Our goal is to show that 
\begin{equation}\label{eq:u_super_soln}
	 \partial_t \varphi(t_0,x_0) + \int_\R \left[  e^{\frac{\sign(h) f( h )}{f'(0)}\partial_x \varphi(t_0,x_0)} - 1 \right] J(h)  \, \dd h + 1 \geq 0.
\end{equation}
Fix any $M>0$.  We eventually take the limit $M \to \infty$.  Using the definition of  $\underline u$ and classical arguments (see \cite{barles1994}), we find $r>0$ and a sequence $(t_\e,x_\e)$ such that $\uep-\varphi$ has a minimum at $(t_\e,x_\e)$ in $(t_0-r,t_0 + r)\times B_{2M}(x_0)$ and such that $(t_\e,x_\e)$ converges to $(t_0,x_0)$ as $\epsilon\to0$ after passing to a sub-sequence, which we denote the same way, if necessary.  Since $u_\e$ is a solution of \cref{eq:uep_mut},
\begin{equation}\label{eq:half_relaxed_prelim}
	\partial _t \varphi(t_\epsilon,x_\epsilon) + \int_\R \left[e^{-\frac{1}{\e}\left(\uep\left( t_\eps , x_\eps -  \sign(h)f^{-1}( \e f(h))\right)-\uep(t_\epsilon,x_\epsilon) \right)} -1\right]J(h) \dd h  + 1 \geq 0.
\end{equation}

The proof now hinges on estimating the integral  in~\cref{eq:half_relaxed_prelim}.  By construction of $(t_\e,x_\e)$ we have
\[
	u_\e(t_\e, x_\e) - \varphi(t_\e,x_\e) \leq u_\e(t,x) - \varphi(t,x)
\]
for all $(t,x)\in (t_0-r,t_0+r)\times B_{2M}(x_0)$.  Also, notice that $|x_0 - x_\e| \leq M$ and $\vert \sign(h)f^{-1}( \e f( h )) \vert \leq M$ for all $h \in [0,M]$ for $\e$ small enough.  Hence, we have 
\begin{equation}\label{eq:I_splitting}
\begin{split}
\int_\R &e^{-\frac{1}{\e}\left(\uep\left( t_\eps , x_\eps -  \sign(h)f^{-1}( \e f(h))\right)-\uep(t_\eps , x_\eps) \right)} J(h) \dd h\\
 			&\qquad\leq\int_{[-M,M]} e^{-\frac{1}{\e}\left(\varphi\left( t_\eps , x_\eps -  \sign(h)f^{-1}( \e f(h))\right)-\varphi(t_\eps , x_\eps) \right)} J(h) \dd h \\
 			 &\qquad\qquad\qquad + \int_{\R\setminus[-M,M]} e^{-\frac{1}{\e}\left(\uep\left( t_\eps , x_\eps -  \sign(h)f^{-1}( \e f(h))\right)-\uep(t_\eps , x_\eps) \right)} J(h) \dd h. \\
 			&\qquad:= I_1^{(M,\e)} + I_2^{(M,\e)}.
\end{split}
\end{equation}

First, we address the integral set on $[-M,M]$, which we denote $I_1^{(M,\e)}$. Since $h$ lies in a bounded set in this integral, $\varphi$ is $\mathcal{C}^1$, and $\lim_{\e \to 0}(t_\e,x_\e) = (t_0,x_0)$, we have
\begin{equation*}
\lim_{\e \to 0} \; \frac1\epsilon\left(\varphi\left( t_\eps , x_\eps -  \sign(h)f^{-1}( \e f(h))\right)-\varphi(t_\eps , x_\eps) \right) = \frac{\sign(h) f(h)}{f'(0)} \partial_x \varphi(t_0,x_0)
\end{equation*}  
uniformly in $[-M,M]$.
Hence, we obtain
\begin{equation}\label{eq:I_1_e_to_zero}
	\lim_{\e \to 0} I^{(M,\e)}_1
		= \int_{[-M,M]} {  e^{\frac{\sign(h) f(h)}{f'(0)}\partial_x \varphi} J(h)  \, \dd h}.
\end{equation}

Next, we address the integral set on $\R\setminus[-M,M]$, which we denote $I_2^{(M,\e)}$.  Using estimate~\cref{ineq:estim_ueph} from~\Cref{lem:estim_uep} on $u_\e$, we have that, for all $h\in\R$,
\begin{equation*}
\uep\left( t_\eps , x_\eps -  \sign(h)f^{-1}( \e f(h))\right)-\uep(t_\eps , x_\eps) \geq - A f(  \sign(h)f^{-1}( \e f(h))  ) = -A \eps f(h),
\end{equation*}
so that
\begin{equation*}
	I_2^{(M,\e)} \leq \int_{\R\setminus[-M,M]} e^{A f(h)} J(h) \dd h = 2\int_{M}^\infty e^{-(1-A) f(h)} \dd h.
\end{equation*} 
Recall, from~\cref{cond:uep0_mutation}, that $A < 1 - 1/\mu$ where $\mu = \liminf_{x\to\infty}x f'(x) > 1$.  This implies that the integrand above is integrable.  Indeed, fix $\alpha_A := \mu/2 - 1/(2 - 2A).$ An easy computation using only that $A < 1 - 1/\mu$ shows that $\alpha_A > 0$. Then from~\Cref{hyp:J} on the kernel $J$, we have that, if $M$ is sufficiently large, $f(h) \geq (\mu - \alpha_A) \ln(\vert h \vert) - C$ for some constant $C>0$, depending only on $f$ and $A$, and all $|h| > M$. Then
\begin{equation*}
\begin{array}{ccl}
\displaystyle I_2^{(M,\e)} 
\leq 2\int_{M}^\infty e^{-(1-A) f(h)} \dd h
	& \leq & \ds 2\int_M^\infty e^{-(1-A) (\mu - \alpha_A) \ln(\vert h \vert) + (1-A)C} \dd h \\
& \leq & \ds 2e^{(1-A)C}\int_{M}^\infty \vert h \vert^{-(1-A)(\mu-\alpha_A)} \dd h.
\end{array}
\end{equation*}
By our choice of $\alpha_A$, it follows that $(1-A)(\mu - \alpha_A) > 1$. Hence the right hand side tends to $0$ as $M\to\infty$, uniformly in $\e$.  

From the estimates~\cref{eq:I_splitting} and \eqref{eq:I_1_e_to_zero}, we can pass to the limit $\e\to 0$ and then $M\to \infty$ in~\cref{eq:half_relaxed_prelim} to obtain the inequality~\cref{eq:u_super_soln}.
This concludes the proof that $\underline u$ is a viscosity super--solution to~\cref{eq:ulim} as desired.

\

In order to show that $\overline u$ is a viscosity sub--solution to~\eqref{eq:ulim}, the steps are almost identical.  The only difference being that one must deal with the term $\exp(-{u_\e}/{\e})$. However, this is easily dealt with by splitting into cases when $u = 0$ and when $u > 0$.  As such, we omit the proof.
\end{proof}

We now move on to the proof of \Cref{theo:mut} (ii).
\begin{proof}[{\bf Proof of \Cref{theo:mut}(ii)}]
 The first step is to state and prove that $\underline u$ and $\overline u$ satisfy related initial conditions so that we may apply the comparison principle to conclude that $\overline u \leq \underline u$. As before, we detail the proof for $\ulbar$ but the proof for $\overline{u}$ is very  similar. The initial condition is
\begin{equation}\label{eq:inicondsuper}
\max\left\{ \min\left\{ \partial_t \ulbar - \left( 1- \int_\R{  e^{\frac{\sign(h) f(h)}{f'(0)}\partial_x \ulbar} J(h)  \, \dd h} \right) +1, \ulbar \right\rbrace,  \ulbar - u^0 \right\} \geq 0,
\end{equation}
on $\left\{ t = 0 \right\} \times \R$ in the viscosity sense,  where $u^0$ is the limit as $\epsilon\to0$ of the initial data sequence $(u_\epsilon^0)_\epsilon.$
To prove the inequality~\eqref{eq:inicondsuper}, let $\phi \in \mathcal{C}^1 \left( (0,+\infty) \times \R \right)$ be a test function such that $\ulbar  - \phi$ has a strict global minimum at $(t_0 = 0,x_0)$. We now prove that either
\begin{equation*} 
\ulbar(0,x_0)\geq u^0(x_0)
\end{equation*}
or 
\begin{equation*}
\partial_t \phi - \left( 1- \int_\R{  e^{\frac{\sign(h) f(h)}{f'(0)}\partial_x \phi} J(h)  \, \dd h} \right) +1 \geq 0 \ \hbox{ on } \left\{ t = 0, \right\} \times \R  \textrm{ and : } \ulbar(0,x_0) \geq 0.
\end{equation*}
Suppose that $\ulbar(0,x_0) < u^0(x_0)$. The argument now starts similarly as in the proof above. By the definition of the lower half--relaxed limit, there exists a sequence $\left( t_{\eps},x_{\eps}\right)$ of minimum points of $u_{\eps}-\psi$ satisfying $(t_{\eps},x_{\eps})\to(0,x_0)$ as $\epsilon\to0$. We first claim that there exists a sub--sequence $(t_{\eps_{k}},x_{\eps_{k}})_k$ of the above sequence, with $\eps_k\to 0$ as $k\to \infty$, such that $t_{\eps_{k}} > 0$, for all $k$.

Suppose that this is not true. Then, for $\epsilon$ small enough, $t_\epsilon=0$ and thus $u_\epsilon-\phi$ has a local minimum at $\left(0,x_{\eps}\right)$. It follows that, for all $(t,x)$ in some neighborhood of $(0,x_{\eps})$, we have
\begin{equation*}
u_{\eps}^0\left( x_{\eps}\right) - \phi\left( 0,x_{\eps}\right) 
\leq u_{\eps}\left( 0,x_{\eps}\right) - \phi\left( 0,x_{\eps}\right) \leq u_{\eps}\left( t , x  \right) - \phi\left( t , x\right).
\end{equation*}
Taking the lower half--relaxed limit as $\epsilon\to0$ and $(t,x)\to(0,x_0)$ on the right hand side of the above inequality, we obtain
\begin{equation*}
u^0(x_0) - \phi \left( 0 , x_0 \right) \leq \ulbar(0,x_0) - \phi \left( 0 , x_0 \right).
\end{equation*}
This contradicts our assumption $\ulbar(0,x_0) < u^0(x_0).$

Hence, there exists a sub--sequence $(t_{\eps_{k}},x_{\eps_{k}})_k$ such that $t_{\eps_{k}} > 0$, for all $k>0$. We can reproduce the same argument as in the proof of \Cref{lem:half_relaxed_limits} above to conclude that~\cref{eq:inicondsuper} holds.

We are now ready to conclude the proof of (ii). Due to standard arguments of viscosity solutions, see~\cite{barles1990comparison,barlesevanssouganidis}, we know that equation~\cref{eq:ulim} has a comparison principle for possibly discontinuous viscosity solutions.  As such, \Cref{lem:half_relaxed_limits} implies that $\overline u \leq \underline u$. 	
On the other hand, we recall that,  $\overline u \geq \underline u$ by construction.  It follows that $\overline u = \underline u$, which in turn implies that $\uep$ converges locally uniformly to a function $u$ satisfying the equation~\cref{eq:ulim}.  Moreover $u$ inherits the gradient bound~\cref{lem:estim_uep}
\begin{equation}\label{ineq:estim_u}
	\left\| \partial_x  u\right\|  _{L^\infty((0,\infty)\times\R )} \leq Af'(0).
\end{equation}
This concludes the proof.
\end{proof}

\subsection{Proof of \Cref{prop:convnullset}: convergence of $\nep$}\label{sec:conv_nulset_mut}
We now return to the behavior of $\nep$ as $\epsilon\to0.$
\begin{proof}[{\bf Proof of \Cref{prop:convnullset}}]\

\medskip

\noindent {\bf \# Convergence on the positive set.} 
Fix any $(t_0,x_0) \in \{(t,x) \in \R\times (0,\infty): u(t,x) > 0\}$.  Since $u_\epsilon$ converges locally uniformly to $u$, then on a small ball around $(t_0,x_0)$, there exists $\alpha>0,$ such that $u_\epsilon(t,x) \geq \alpha > 0$ for all $\epsilon$ sufficiently small. Using the Hopf--Cole transform, we see that for all $(t,x)$ in a small ball around $(t_0,x_0)$,
\[
	\nep(t,x) = e^{-\frac{u_\epsilon(t,x)}{\epsilon}}
		\leq e^{-\frac{\alpha}{\epsilon}}.
\]
 Taking the limit $\epsilon \to 0$ clearly yields the convergence of $\nep$ to zero.  Hence $\nep(t,x)$ converges to $0$ locally uniformly on $\{(t,x) \in \R\times(0,\infty): u(t,x) > 0\}$.

\

\noindent{\bf \# Convergence on the null set.} 
We next consider the case when $(t_0,x_0)$ is an element of the interior of $\{(t,x) \in \R\times (0,\infty): u(t,x) = 0 \}$.  Take $r$ sufficiently small so that $u$ vanishes on the ball $B_r(t_0,x_0)$.  Consider the test function
\[
	\phi(t,x) = \frac{2A f'(0)}{r} \vert x-x_0 \vert^2 + \vert t-t_0\vert^2, \ \hbox{ for all } t>0 \hbox{ and } x\in\R;
\]
Due to the finite difference bound that $u$ inherits from~\Cref{lem:estim_uep}, it is easy to check that $(u-\phi)$ has a strict local  maximum at $(t_0,x_0)$. In addition, the function $x\mapsto (u-\phi)(t_0,x)$ has a strict global maximum at $x_0.$ Indeed, we have that
\[
	u(t_0,x) = u(t_0,x) - u(t_0,x_0)
		\leq A f(x-x_0)
		\leq A |x-x_0| f'(\xi),
\]
for some $\xi \in [0,|x-x_0|]$.  Since $f$ is concave, then $f'(\xi) \leq f'(0)$.  Consider first $x$ such that $|x-x_0|>r$.  In this case, we have that
\[
	u(t_0,x) \leq A|x-x_0| f'(0)
		\leq \frac{A f'(0)}{r} |x-x_0|^2
		< \phi(t_0,x).
\]
Hence we have that $u(t_0,x) - \phi(t_0,x) < 0 = u(t_0,x_0) - \phi(t_0,x_0)$ for all $x$ such that $|x-x_0| > r$.  On the other hand, if $0 < |x-x_0|\leq r$ then $u(t_0,x) = 0$ and we have that $u(t_0,x) - \phi(t_0,x) < 0 = u(t_0,x_0) - \phi(t_0,x_0)$.  In both cases, we see that $x_0$ is a strict global maximum in $x$ of $u - \phi$ at time $t_0$.

Since $u_\epsilon$ converges locally uniformly to $u$, for $\epsilon$ small enough, we can construct sequence of points $(t_\epsilon,x_\epsilon)$ such that $(t_\e, x_\e)$ is the location of a maximum  of $u_\epsilon - \phi$ in $B_r(t_0,x_0)$ and  $(t_\epsilon,x_\epsilon) \to (t_0,x_0)$ as $\epsilon\to0.$  In addition, arguing as above, the function $x\mapsto (u_\epsilon-\phi)(t_\epsilon,x)$ has a global maximum in $x_\epsilon$. This gives us the inequalities, for all $x \in \R$,
\begin{equation}\label{eq:nice_test_fn}
	u_\epsilon(t_\epsilon,x) - u_\epsilon(t_\epsilon, x_\epsilon)
		\leq \phi(t_ \epsilon,x) - \phi(t_\epsilon, x_\epsilon),
		\text{ and } 
	\partial_t u_\e (t_\eps, x_\e) = 2(t_\e - t_0).
\end{equation}
Since $u_\epsilon$ solves~\cref{eq:uep_mut}, we deduce from~\cref{eq:nice_test_fn} that for $(t_\e, x_\e)$,
\begin{equation}\label{eq:n_e_convergence}
\begin{split}
	n_\e(t_\epsilon,x_\epsilon) -1 
	    & = \partial_t u_\e(t_\epsilon,x_\epsilon)\! - \!\left(1\! - \!\int_\R \!\!\!e^{-\frac{1}{\e}\left(\uep\left( t_\eps , x_\eps -    \sign(h)f^{-1}( \e f(h))\right)-\uep(t_\e, x_\e) \right)} J(h) \dd h\right) \\
	    & = 2(t_\e - t_0) - \left(1 - \int_\R \!\!\!e^{-\frac{1}{\e}\left(\uep\left( t_\eps , x_\eps -  \sign(h)f^{-1}( \e f(h))\right)-\uep(t_\e, x_\e) \right)} J(h) \dd h\right).
\end{split}
\end{equation}
Arguing as above and using~\cref{eq:nice_test_fn} with the integral term in~\cref{eq:n_e_convergence}, it follows that
\begin{multline*}
\lim_{\eps \to 0} \left(1 - \int_\R e^{-\frac{1}{\e}\left(u_\e\left( t_\eps , x_\eps -  \sign(h)f^{-1}( \e f(h))\right)-u_\e(t_\epsilon,x_\epsilon) \right)} J(h) \dd h\right) \\ \leq  1- \int_\R{  e^{\frac{\sign(h) f(h)}{f'(0)}\partial_x \phi(t_0,x_0)} J(h)  \, \dd h} = 0.
\end{multline*}
Here, the last equality used the explicit expression of $\phi$, which gives $\partial_x \phi(t_0,x_0) = 0$.  The above yields, along with~\cref{eq:n_e_convergence}
\begin{equation}\label{chris34}
	\liminf_{\e\to0} n_\e(t_\e,x_\e) \geq 1.
\end{equation}

Since $t_\epsilon\to t_0$ as $\epsilon\to0,$  we may conclude that $\ds\liminf_{\e \to 0} n_\e(t_\e, x_\e) \geq 1$.  On the other hand, we have that
\begin{equation*}\label{eq:n_e_to_1}
	\liminf_{\e \to 0} n_\e(t_0,x_0)
		\geq \liminf_{\e\to0} n_\e(t_\e,x_\e) e^{\frac{\phi(t_\e,x_\e)}{\e}}
		\geq 1,
\end{equation*}
where the first inequality is due to~\cref{eq:nice_test_fn} and the second is due to the non-negativity of $\phi$ along with~\cref{chris34}.
Using that $n_\e  \leq 1$, we conclude that $\lim_{\eps \to 0} n_\e(t_0,x_0) = 1$ as claimed.
\end{proof}

\subsection{Proof of \Cref{lem:estim_uep}: the {\em a priori} bounds}\label{sec:a_priori}

The only remaining ingredient is to prove the {\em a priori} bounds on $u^\e$.  We proceed by constructing explicit sub- and super-solutions.

\begin{proof}[{\bf Proof of \Cref{lem:estim_uep}}]
To estimate $u_\e$ from above, we first observe that $\nep$ is positive and bounded by $1$ and $u_\epsilon$ solves  \cref{eq:uep_mut}. Thus
\begin{equation*}
\begin{cases}
\partial_t \uep(t,x) \leq 1, \qquad & \hbox{on } (0,\infty)\times\R\\
\uep(0,\cdot) =\uep^0.
\end{cases}
\end{equation*}
As a consequence, for all $t\geq 0$ and $x\in\R,$ 
\begin{equation*}
\uep(t,x) \leq \uep^0(x) + t.
\end{equation*}

To get a bound from below, we define $\underline s(t,x) = \uep^0(x) - rt$, where $r$ is chosen below.  We prove that $\underline s$ is a super--solution. Using assumption \cref{cond:uep0_mutation}, for all $\epsilon>0$ and  $x \in \R$,
\begin{equation*}
\begin{array}{ll}
\ds \int_{-\infty}^\infty e^{-\frac{1}{\e}\left( \uep^0\left(x - \psi_\eps(h) \right)-\uep^0(x) \right)} J(h) \dd h & \leq \ds \int_{-\infty}^\infty e^{\frac{A}{\e} f\left(\vert \psi_\eps(h) \vert \right)} J(h) \dd h \\[0.3cm]
& =\ds  \int_{-\infty}^\infty e^{A f(h)} J(h) \dd h < + \infty.
\end{array}
\end{equation*}
Define $r := \int_\R e^{Af(h)}dh$.  We deduce that $\underline{s}$ satisfies 
\begin{equation*}
\begin{cases}
\partial_t \underline{s} - \left(1 - \displaystyle\int_\R e^{-\frac{1}{\e}\left(\underline{s}\left( t, x - \psi_\eps(h)  \right)-\underline{s}(t,x) \right)} J(h) \dd h\right) + 1 \leq 0, & \hbox{on } (0,\infty)\times\R,\\
\underline{s}(0,x) \leq \uep^0(x), & x\in\R.
\end{cases}
\end{equation*}
The function $\underline{s}$ is a sub--solution to \cref{eq:uep_mut}.  Hence, for all $t \geq 0$ and $x \in \R$,
\begin{equation*}
- rt \leq u_\e(t,x) - u_\e^0(x) \leq t.
\end{equation*}

To conclude the proof of the lemma, we now prove the inequality on the finite difference of $u_\e$, namely \cref{ineq:estim_ueph}.  To this end, we define for all $t\geq0$, $x\in\R$ and $h\in \R$,
\[
	w_{\e,h}(t,x):= u_\e(t,x) - u_\e(t,x+h).
\]
Then, using~\cref{eq:uep_mut} we see that $w_{\e,h}$ satisfies the equation
\begin{equation*}
\begin{array}{lll}
  	 \partial_t w_{\e,h} + \ds\int_\R & \left( e^{-\frac{1}{\e}\left(\uep\left( t, x - \psi_\eps(h)  \right)-\uep(t,x) \right)} -\right.&  \left. e^{-\frac{1}{\e}\left(\uep\left( t, x  - \psi_\eps(h) +h \right)-\uep(t,x+h) \right)} \right) J(h) \dd h \\
 	& = \left( 1 - e^{\frac{w_{\e,h}}{\eps}} \right)  \nep, \qquad &\hbox{on } \ (0,\infty)\times\R,\\
	 w_{\e,h}(0,x) & = \uep^0(x) -  \uep^0(x+h),  \qquad &\text{for all } x\in\R.
\end{array}
\end{equation*}
This reduces to
\begin{equation*}
\begin{array}{lll}
 \partial_t w_{\e,h} 
 +\ds \int_\R & e^{-\frac{1}{\e}\left(\uep\left( t, x - \psi_\eps(h)  \right)-\uep(t,x) \right)} & \left( 1 - e^{\frac{1}{\e}\left(w_{\e,h}\left( t, x  - \psi_\eps(h) \right)-w_{\e,h}(t,x) \right)} \right) J(h) \dd h \\
 & = \left( 1 - e^{\frac{w_{\e,h}}{\eps}} \right)  \nep,&\hbox{on } \ (0,\infty)\times\R,\\
 w_{\e,h}(0,x) & = \uep^0(x) -  \uep^0(x+h), &\text{for all } x\in\R.
\end{array}
\end{equation*}
We apply the maximum principle to deduce that
\[
	\sup_{x\in\R} w_{\e,h}(t,x)
		\leq \max \left\{0, \sup_{x\in\R} w_{\e,h}(0,x)\right\}.
\]
This implies that for all $x \in \R$ and $h \in \R$,
\begin{multline*}
	u_\e(t,x+h) - u_\e(t,x) \geq \min \left\{0, \inf_{y\in\R} \left( \uep^0(y+h) -  \uep^0(y) \right) \right\} \\ \geq \min \left\{0, - Af(h) \right\} = - Af(h).
\end{multline*}
This finishes the proof.
\end{proof}

\section*{Acknowledgments}

The authors would like to thank Sepideh Mirrahimi for helpful discussions regarding~\cite{SMSM}. 
Part of this work was performed within the framework of the LABEX MILYON (ANR- 10-LABX-0070) of Universit\'e de Lyon, within the program ``Investissements d'Avenir'' (ANR-11- IDEX-0007) and the project NONLOCAL (ANR-14-CE25-0013) operated by the French National Research Agency (ANR). In addition, this project has received funding from the European Research Council (ERC) under the European Unions Horizon 2020 research and innovation programme (grant agreement No 639638).  CH was partially supported by the National Science Foundation Research Training Group grant
DMS-1246999. 

\bibliographystyle{siamplain}
\bibliography{references}	
\end{document}